\documentclass[12pt]{amsart}
\usepackage{amsmath,amssymb,amsthm}
\usepackage{mathtools}
\usepackage{xcolor}
\usepackage{graphicx} 
\usepackage{pgf,tikz}
\usetikzlibrary{arrows}
\usepackage{lipsum}
\usepackage{caption}  
\usepackage{subcaption}
\usepackage{pgfplots}
\usepackage{scalerel}
\usepackage{amsmath}
\usepackage{amssymb}
\usepackage{amsthm}
\usepackage{enumerate}
\usepackage{hyperref}
\hypersetup{colorlinks=true,citecolor=blue,linkcolor=red}
\usepackage{cite}
\usepackage{color}
\usepackage{bbm}
\usepackage[usestackEOL]{stackengine}

\usepackage[a4paper,width=16.5cm,top=3cm,bottom=2.8cm]{geometry}

\newtheorem{theorem}{Theorem}[section]
\newtheorem{proposition}[theorem]{Proposition}
\newtheorem{corollary}[theorem]{Corollary}
\newtheorem{lemma}[theorem]{Lemma}
\newtheorem{definition}[theorem]{Definition}

\theoremstyle{plain}
\newtheorem*{theorem*}{Theorem}

\def\dashint{\,\ThisStyle{\ensurestackMath{%
			\stackinset{c}{.2\LMpt}{c}{.5\LMpt}{\SavedStyle-}{\SavedStyle\phantom{\int}}}%
		\setbox0=\hbox{$\SavedStyle\int\,$}\kern-\wd0}\int}
\newcommand{\rn}{\mathbb{R}^n}
\newcommand{\rd}{\mathbb{R}^d}
\newcommand{\pk}{\mathbb{P}_k}
\newcommand{\hrn}{{H}^1(\mathbb{R}^n)}

\newcommand{\hpo}{{H}^p(\Omega)}
\newcommand{\hprn}{{H}^p(\mathbb{R}^n)}
\newcommand{\mom}{M_{\Omega, \varphi}}

\newcommand{\cinftyc}[1]{C_{0}^{\infty}(#1)}
\newcommand{\hlip}[2]{\dot{\Lambda}^{#1}(\mathbb{R}^{#2})}
\newcommand{\ckrn}{C^k(\rn)}
\newcommand{\linfno}[2]{\|#1\|_{L^{\infty}(#2)}}
\newcommand{\lprn}{{L}^p(\rn)}
\newcommand{\lpon}[1]{\|#1\|_{{L}^p(\Omega)}}
\newcommand{\hpon}[1]{\|#1\|_{\hpo}}
\newcommand{\hprnn}[1]{\|#1\|_{\hprn}}
\newcommand{\lprnn}[1]{\|#1\|_{\lprn}}
\newcommand{\dist}[1]{\text{d}(#1)}
\newcommand{\infn}[1]{\|#1\|_{{L}^{\infty}}}

\newcommand{\bmorn}{{BMO}(\rn)}
\newcommand{\bmornn}[1]{\|#1\|_{{BMO}(\rn)}}
\newcommand{\hrnn}[1]{\|#1\|_{\hrn}}
\newcommand{\wkp}{W^{k,p}}

\newcommand{\lp}{L^p}
\newcommand{\hp}{H^p}

\newcommand{\hlipn}[3]{\|#1\|_{\hlip{#2}{#3}}}

\pgfplotsset{compat=1.18}
\begin{document}

	\title{Extension domains for Hardy spaces}
	\author{Shahaboddin Shaabani}
	\address{Department of Mathematics and Statistics\\Concordia University}
	\email{shahaboddin.shaabani@concordia.ca}
	\date{}
	
	\begin{abstract}
		We show that a proper open subset $\Omega\subset \rn$ is an extension domain for $\hp$ ($0<p\le1$), if and only if it satisfies a certain geometric condition. When $n(\frac{1}{p}-1)\in \mathbb{N}$ this condition is equivalent to the global Markov condition for $\Omega^c$, for $p=1$ it is stronger, and when $n(\frac{1}{p}-1)\notin \mathbb{N}\cup \{0\}$ every proper open subset is an extension domain for $H^p$. It is shown that in each case a linear extension operator exists. We apply our results to study some complemented subspaces of $\bmorn$.
	\end{abstract}	
	\maketitle	
	\section{Introduction}	
	In general for a given function space $X$, the extension problem is to  characterize open subsets $\Omega\subset\rn$, such that there exists a bounded  (linear or nonlinear) operator, extending every element of $X(\Omega)$  to an element of $X(\rn)$. Such domains are called extension domains for $X$.\\
	
	This problem has been studied for several classes of functions like Sobolev spaces $\wkp$ \cite{MR0631089}, functions of bounded mean oscillation $BMO$ \cite{MR0554817}, $VMO$\cite{MR4289249}, etc. Here, we're going to give an answer to the extension problem for Hardy spaces $\hp$ when $0<p\le 1$. As it is well known, this scale of spaces serves as a good substitute for $\lp$ when $0<p\le 1$ \cite{MR1232192}. Unlike $\lp$, being in $\hp$ is not just a matter of ``size'' but also a matter of ``cancellation''. Indeed, these distributions must have vanishing moments up to order $n(\frac{1}{p}-1)$, and this makes the extension problem a bit nontrivial. In Theorem \ref{theorem1}, which is the main result of this paper, we give a characterization of these domains, and below we explain the idea behind the proof.\\
	
	By Miyachi's atomic decomposition theorem, the extension problem is reduced to extending certain elements of $\hpo$, namely $(p,\Omega)$-atoms. These are atoms without any cancellations and are supported on Whitney cubes of $\Omega$. Then since our characterization requires $\Omega^c$ to be thick in some sense, in dilated Whitney cubes, we are able to create the required cancellation. When $p=1$, this can be done by putting a proper constant on $\Omega^c$ in that dilated cube, and for $p=\frac{n}{n+k}$, adding a suitable linear combination of derivatives of Dirac masses supported on $\Omega^c$ will do the job. The latter case involves interpolating data with polynomials of several variables, and we show that with the condition mentioned in Theorem \ref{theorem1}, it is always possible to do that. Also in the absence of these thickness conditions, we are able to construct some functions in the dual of $\hprn$, namely $\bmorn$ and $\hlip{k}{n}$, which are supported on $\Omega$ and have small norms yet large averages. This prevents any bounded extension of $(p,\Omega)$-atoms. Finally, when $n(\frac{1}{p}-1)$ is not an integer, it can be shown that every open set is an extension domain.\\
	\section{Notation and Definitions}	
	We begin by fixing some notation that we use throughout the paper. We use $d(x,H)$ for the distance of $x$ from $H$, and $d(x)=d(x,\Omega^c)$. The space of smooth functions with compact support in $\Omega$ is denoted by $\cinftyc{\Omega}$, distributions in $\Omega$ by $\mathcal{D}'(\Omega)$, and $\left\langle	
	,\right\rangle$ is used for pairing distributions and test functions. $\ckrn$ is used for the space of functions with bounded continuous derivatives up to order $k$ with the usual sup-norm, and ${\|P\|}_{C^{k}(E)}$ means that we take the $\sup$-norms only on $E$. Also the space of polynomials of degree no more than $k$ in $\rn$ is denoted by $\pk$, as usual $\partial^\alpha$  is used for partial derivatives, $|\alpha|$ for the total degree of $\alpha$, and other famous notations is used as well. By a cube $Q$ we mean a  cube with sides parallel to the axis and we use $l(Q)$ for its side length. We use $aQ$ or $aB$ to denote the concentric dilation of a cube $Q$ or a ball $B$. Also $\mathcal{W}(\Omega)$ is used for the family of Whitney cubes of $\Omega$ (for the definition of Whitney cubes see \cite{MR2445437} App.J, or \cite{MR1060724}). Moreover, the useful symbols $\lesssim$, $\gtrsim$ and $\approx$ are used to avoid writing unimportant constants usually depending on $k$ and $n$.\\
	
	Next, let us recall definitions of the dual of Hardy spaces $\hprn$. The space of functions of bounded mean oscillation $\bmorn$ consists of all locally integrable functions $f$ such that  
	\[
	\sup_{Q}\dashint_Q|f-\dashint_Qf|<\infty,
	\]
	where in the above the $\sup$ is taken over all cubes $Q$. Modulo constants,  $\bmorn$ is a Banach space with the above quantity as norm, and as it is well known it is the dual space of $\hrn$ (see \cite{MR0807149} Ch.3, \cite{MR2463316} Ch.3,\cite{MR1232192} Ch.4).\\
	
	When $0<p<1$, the dual space of $\hprn$ coincides with the homogeneous Lipschitz space $\hlip{n(\frac{1}{p}-1)}{n}$, which we now define. For a function $f$ defined on $\rn$ and $h\in \rn$, the forward difference operator $\Delta_h$ is defined by $\Delta_h(f)(x):=f(x+h)-f(x)$, and for any positive integer $k$ let $\Delta^k_h(f):=\Delta_h(\Delta_h...(\Delta_h(f)))$ be the $k$-th iteration of $\Delta_h$.
	Then for any positive real number $\gamma$, the space of homogeneous Lipschitz functions $\hlip{\gamma}{n}$ consists of all locally integrable functions $f$ such that 
	
	\[
	\sup_{h\ne0}\linfno{\frac{\Delta^{[\gamma]+1}_h(f)}{|h|^{\gamma}}}{\rn}<\infty.
	\]
	Modulo polynomials of degree no more than $\gamma$, $\hlip{\gamma}{n}$ is also a Banach space with the above quantity as norm, and as we already mentioned, for $0<p<1$ the dual space of $\hprn$ is $\hlip{n(\frac{1}{p}-1)}{n}$ (see \cite{MR0807149} Ch.3 \cite{MR2463316} Ch.1). Here we want to make this clear that $\hlip{\gamma}{n}$ is referred to by different names. For instance, when $\gamma=1$ it's called Zygmund space, or it is identical with the homogeneous Besov space $\dot{B}_{\infty,\infty}^\gamma(\rn)$ (see \cite{MR4567945} Ch.17).\\
	
	Now we turn to the definition of Hardy spaces on domains. Let $\Omega$ be a proper open subset of $\rn$ and $f$ a distribution in $\mathcal{D}^{'}(\Omega)$. Also, let $\varphi$ be a smooth function supported in the unit ball of $\rn$ such that $\int \varphi =1$, and for $t>0$, let $\varphi_t(x)=t^{-n}\varphi(\frac{x}{t})$. Then, the maximal function of $f$ associated to $\Omega$ and $\varphi$, denoted by $M_{\Omega, \varphi}$, is defined by
	
	\[
	\mom(f)(x):=\sup\{ \left\langle	
	f,\varphi_t(x-.)\right\rangle  |0<t<\text{dist}(x,\Omega^c)\} \quad \text{for}  \quad x\in \Omega.
	\]
	\begin{definition}
		For $0<p\le 1$, $\hpo$ is the space of all distributions $f$ on $\Omega$ such that $\mom(f)$ is in ${L}^p(\Omega)$. This defines a linear space, which is independent of $\varphi$ and can be quasi normed (normed when $p=1$) by setting $\hpon{f}=\lpon{\mom(f)}$. For different choices of $\varphi$, these quasi-norms are equivalent.
	\end{definition}		
	The Hardy space $\hpo$ was introduced by Miyachi in \cite{MR1060724}, and an atomic decomposition theorem was proved for it there. As a result, many fundamental properties of these spaces, including characterization of the dual space, density of $\cinftyc{\Omega}\cap \hpo$, complex interpolation and extension problem were studied although the latter was not fully resolved there. Meaning the following sufficient condition was obtained but as we will see it's not necessary.
	\begin{theorem}[Miyachi]
		Let $\Omega$ be a proper open subset of $\rn$. Suppose there exists a constant $a>1$ such that for every $x\in \Omega$ there is $x'\in \Omega^c$ with $d(x,x')< ad(x,\Omega^c)$ and $d(x',\Omega)>a^{-1}d(x,\Omega^c)$. Then there exists a bounded linear extension operator from $\hpo$ to $\hprn$.
	\end{theorem}

	We continue this section by recalling definitions of two types of atoms. 
	The first kind are the usual atoms appearing in the atomic decomposition of $\hprn$.
	\begin{definition}
		A bounded function $f$ is called a $p$-atom, if there exists a ball $B$ (the supporting ball of $f$) such that $\text{supp}(f)\subset B$, $\infn{f}\le |B|^{-\frac{1}{p}}$ (size condition), and $\left\langle f , P \right\rangle =0$ for all polynomials with degree no more than $n(\frac{1}{p}-1)$ (cancellation condition).
	\end{definition}
	However, there's another type of atom which appears in the atomic decomposition of $\hpo$. These atoms do not need to have any kind of cancellations and are called $(p,\Omega)$-atoms.
	\begin{definition}
		Let $\Omega$ be a proper open subset of $\rn$. We say that a bounded function $g$ is a $(p,\Omega)$-atom if it is supported on some Whitney cube $Q$ of $\Omega$, and $\infn{g}\le |Q|^{-\frac{1}{p}}$.
	\end{definition}
	In \cite{MR1060724}, it's shown that such atoms are elements of $\hpo$ and are uniformly bounded in this space. After fixing these definitions let us recall the atomic decomposition theorem for Hardy spaces on domains.
	
	\begin{theorem}[Miyachi]\label{miyachi}
		Let $\Omega$ be a proper open subset of $\rn$, $0<p\le 1$ and $f\in \hpo$, then there exist sequences of nonnegative numbers $\{\lambda_i\} , \{\mu_j\}$, a sequence of $p$-atoms $\{f_i\}$, and a sequence of 
		$(p, \Omega)$-atoms $\{g_j\}$, such that the latter is depended linearly on $f$, each $f_i$ is supported in $\Omega$, and 
		\begin{equation}\label{miachidecomposition}
			f=\sum \lambda_i f_i + \sum \mu_j g_j \quad \text{with} \quad \hpon{f} \approx \left(\sum \lambda^p_i +\sum \mu^p_j\right)^{\frac{1}{p}}.
		\end{equation}
	\end{theorem}
	In the above, first a suitable Whitney decomposition of the domain is fixed and then roughly speaking, $g_j$s are obtained by projecting $f$ onto the space of polynomials localized on Whitney cubes of $\Omega$. This is why $g_j$ depends linearly on $f$, but this is not the case for $f_j$s. Also, it follows from the construction that if $f$ is compactly supported in $\Omega$, then the support of all $f_js$ and $g_js$ are contained in a probably larger compact subset of $\Omega$ \cite{MR1060724}. It should be noted here that even if $f$ is a $(p,\Omega)-$atom, the above decomposition gives us some $p$-atoms and $(p,\Omega)$-atoms which are probably different from $f$.\\
	
	In \cite{MR1060724}, by using this powerful theorem, the extension problem for distributions in $\hpo$ was reduced to the especial case of extending $(p,\Omega)$-atoms. Although the following lemma is used there, but it's not stated, so we decided to state it here as a lemma.
	\begin{lemma}\label{lemma0}
		For an open set $\Omega$ to be an $H^p$ extension domain it is necessary and sufficient that for every $(p,\Omega)$-atom $g$ there exists an extension $\tilde{g} \in \hprn$ such that $\hprnn{\tilde{g}} \lesssim 1$. Moreover, if this extension is linear in terms of $g$ then there is a bounded linear extension operator on $\hpo$.
	\end{lemma}
	\begin{proof}
		The necessity of the above condition follows directly from the definition of extension domains and the fact that $(p,\Omega)$-atoms have bounded norms.\\
		For the sufficiency, take a distribution $f \in \hpo$, decompose it into atoms as in Theorem \ref{miyachi} and extend it in the following way:
		
		\[
		\tilde{f} =\sum \lambda_i \tilde{f_i} + \sum \mu_j \tilde{g_j},
		\]
		where $\tilde{f_i}$ is the extension of $f_i$ by considering it to be zero outside of $\Omega$. Since $\text{supp}(f_i)\subset \Omega$, $\tilde{f}$ is a well-defined distribution which belong to $\hprn$ and its norm is bounded by	
		\[
		\hprnn{\tilde{f}}^p \le \sum \lambda^p_i \hprnn{\tilde{f_i}}^p + \sum \mu^p_j \hprnn{\tilde{g_j}}^p \lesssim \sum \lambda^p_i +\sum \mu^p_j\lesssim\hpon{f}^p.
		\]		
		So this defines a bounded extension operator on $\hpo$. Whenever the extension of $(p,\Omega)$-atoms is linear, the above operator is also linear on compactly supported distributions in $\hpo$. To see this, we note that $\sum \mu_j \tilde{g_j}$ is linear in terms of $f$ and $\sum \lambda_i \tilde{f_i}=(f-\sum \mu_jg_j)\chi_{\Omega}$, which is true because the support of $f$, $f_i$s and $g_j$s all are contained in a compact subset of $\Omega$. This implies that $\sum \lambda_i \tilde{f_i}$ is also linear in $f$. Now, the operator $f\to \tilde{f}$ is densely defined and bounded, so it has a unique extension to  $\hpo$, which is the desired linear extension operator.
		
	\end{proof}

	\section{Main result}	
	\begin{definition}
		Let $E\subset \rn$ and $\nu$ be a unit vector, then the width of $E$ in direction $\nu$, $w(E,\nu)$, is defined as the smallest number $\delta>0$ such that $E$ lies entirely between two parallel hyperplanes with normal $\nu$ and of distance $\delta$ from each other. Also we define the width of $E$, $w(E)$, to be
		\[
		w(E):=\inf_{|\nu|=1}w(E,\nu).
		\]
		When $n=1$, by $w(E)$ we simply mean the diameter of $E$.
	\end{definition}	
	
	\begin{theorem}\label{theorem1}
		Let $0<p\le 1$ then a proper open subset $\Omega\subset \rn$ is an $H^p$ extension domain if and only if
		\begin{itemize}
			\item 
			
			For $p=1$ there exist constants $a>1$ and $\delta>0$, such that for every $x \in \Omega$

			\begin{equation}\label{h1cc}
				\frac{|\Omega^c\cap B(x,a\dist{x})|}{|B(x,a\dist{x})|}>\delta.
			\end{equation}

			\item
			For $p=\frac{n}{k+n}$, $k=1,2,...$, there exist positive constants $a>1$ and $\delta>0$ such that for every $x \in \Omega$,

			\begin{equation}\label{key}
				\frac{w(\Omega^c\cap B(x,ad(x)))}{d(x)}>\delta.
			\end{equation}

			\item
			And if $p\neq \frac{n}{k+n}$, $k=0,1,2,3...$ then every proper open subset is an extension domain.
		\end{itemize}
		
		Moreover in all cases the extension operator can be taken to be linear.
	\end{theorem}
	
	In the above, the condition \eqref{h1cc} is stronger than \eqref{key} because if \eqref{h1cc} holds and $\Omega^c\cap B(x,ad(x))$ lies between two parallel hyperplanes of distance $\delta'd(x)$, then $$|\Omega^c\cap B(x,ad(x))|\le \delta'd(x)(2ad(x))^{n-1},$$ which implies that $\delta'\ge2^{1-n}c_n a\delta $, where $c_n$ is the volume of the unit ball of $\rn$. So \eqref{key} holds with the same $a$ and $\delta$ replaced by $2^{1-n}c_n a\delta$. Also the example below shows that these two conditions are different.\\ 
	
	For a closed unit cube $Q=I_1\times...\times I_n$ in $\rn$, let $C_j$ be the usual Cantor middle third subset of $I_j$ and call $C_1\times...\times C_n$ the Cantor dust in $Q$. Now consider the grid of unit cubes in $\rn$ with vertices on $\mathbb{Z}^n$, and from each cube in this grid remove its Cantor dust and call the remaining set $\Omega$. For $n\ge 2$, $\Omega$ is connected and open with $|\Omega^c|=0$ so it doesn't satisfy \eqref{h1cc} for any choice of $a$ and $\delta$ but it satisfies \eqref{key}. To see this take $x=(x_1,...,x_n)\in \Omega$ then it lies in some cube $Q$ of the grid and without the loss of generality assume $Q=[0,1]^n$. Now suppose $d(x,\Omega^c)\approx \max_{1\le j\le n}|x_j-y_j|=|x_1-y_1|$ for some $y=(y_1,...,y_n)$ with all $y_j\in C$, where $C$ is the Cantor set in $[0,1]$. From the ternary expansion of $x_j$ and $y_j$, it's clear that $|x_j-y_j|\approx 3^{-k_j}$ where $k_j$ is the first place that their digits are different. So if in the expansion of $y_j$ we change the $k_1$-th digit from 0 to 2 or from 2 to 0, we get another point $y_j'$ with $|y_j-y_j'|\approx 3^{-k_1}$. Then note that the cube $Q'=[y_1,y_1']\times...\times [y_j,y_j']$ has all its vertices on $\Omega^c$ and $d(x,Q') \approx d(x,\Omega^c)\approx l(Q')$. This shows that in every direction, the width of $\Omega^c$ in a dilation of $B(x,d(x,\Omega^c))$ is comparable to $d(x,\Omega^c)$ and \eqref{key} holds (see figure below).
	\begin{figure}[h]
		\definecolor{zzttqq}{rgb}{0.6,0.2,0}
		\definecolor{uququq}{rgb}{0.25,0.25,0.25}
		\begin{tikzpicture}[line cap=round,line join=round,>=triangle 45,x=0.17cm,y=0.17cm]
			\clip(-18.89,-4.46) rectangle (46.74,29.92);
			\fill[color=zzttqq,fill=zzttqq,fill opacity=0.1] (6,21) -- (6,20) -- (7,20) -- (7,21) -- cycle;
			\fill[color=zzttqq,fill=zzttqq,fill opacity=0.1] (8,21) -- (9,21) -- (9,20) -- (8,20) -- cycle;
			\fill[color=zzttqq,fill=zzttqq,fill opacity=0.1] (7,19) -- (7,18) -- (6,18) -- (6,19) -- cycle;
			\fill[color=zzttqq,fill=zzttqq,fill opacity=0.1] (18,19) -- (19,19) -- (19,18) -- (18,18) -- cycle;
			\fill[color=zzttqq,fill=zzttqq,fill opacity=0.1] (21,21) -- (20,21) -- (20,20) -- (21,20) -- cycle;
			\fill[color=zzttqq,fill=zzttqq,fill opacity=0.1] (24,21) -- (24,20) -- (25,20) -- (25,21) -- cycle;
			\fill[color=zzttqq,fill=zzttqq,fill opacity=0.1] (26,21) -- (26,20) -- (27,20) -- (27,21) -- cycle;
			\fill[color=zzttqq,fill=zzttqq,fill opacity=0.1] (24,19) -- (24,18) -- (25,18) -- (25,19) -- cycle;
			\fill[color=zzttqq,fill=zzttqq,fill opacity=0.1] (26,19) -- (26,18) -- (27,18) -- (27,19) -- cycle;
			\fill[color=zzttqq,fill=zzttqq,fill opacity=0.1] (24,27) -- (24,26) -- (25,26) -- (25,27) -- cycle;
			\fill[color=zzttqq,fill=zzttqq,fill opacity=0.1] (26,27) -- (26,26) -- (27,26) -- (27,27) -- cycle;
			\fill[color=zzttqq,fill=zzttqq,fill opacity=0.1] (24,25) -- (24,24) -- (25,24) -- (25,25) -- cycle;
			\fill[color=zzttqq,fill=zzttqq,fill opacity=0.1] (26,25) -- (26,24) -- (27,24) -- (27,25) -- cycle;
			\fill[color=zzttqq,fill=zzttqq,fill opacity=0.1] (21,27) -- (21,26) -- (20,26) -- (20,27) -- cycle;
			\fill[color=zzttqq,fill=zzttqq,fill opacity=0.1] (19,26) -- (18,26) -- (18,27) -- (19,27) -- cycle;
			\fill[color=zzttqq,fill=zzttqq,fill opacity=0.1] (18,25) -- (18,24) -- (19,24) -- (19,25) -- cycle;
			\fill[color=zzttqq,fill=zzttqq,fill opacity=0.1] (20,25) -- (20,24) -- (21,24) -- (21,25) -- cycle;
			\fill[color=zzttqq,fill=zzttqq,fill opacity=0.1] (8,19) -- (8,18) -- (9,18) -- (9,19) -- cycle;
			\fill[color=zzttqq,fill=zzttqq,fill opacity=0.1] (19,21) -- (18,21) -- (18,20) -- (19,20) -- cycle;
			\fill[color=zzttqq,fill=zzttqq,fill opacity=0.1] (20,19) -- (20,18) -- (21,18) -- (21,19) -- cycle;
			\fill[color=zzttqq,fill=zzttqq,fill opacity=0.1] (1,9) -- (0,9) -- (0,8) -- (1,8) -- cycle;
			\fill[color=zzttqq,fill=zzttqq,fill opacity=0.1] (1,7) -- (0,7) -- (0,6) -- (1,6) -- cycle;
			\fill[color=zzttqq,fill=zzttqq,fill opacity=0.1] (2,9) -- (2,8) -- (3,8) -- (3,9) -- cycle;
			\fill[color=zzttqq,fill=zzttqq,fill opacity=0.1] (2,7) -- (2,6) -- (3,6) -- (3,7) -- cycle;
			\fill[color=zzttqq,fill=zzttqq,fill opacity=0.1] (6,9) -- (6,8) -- (7,8) -- (7,9) -- cycle;
			\fill[color=zzttqq,fill=zzttqq,fill opacity=0.1] (8,9) -- (8,8) -- (9,8) -- (9,9) -- cycle;
			\fill[color=zzttqq,fill=zzttqq,fill opacity=0.1] (6,7) -- (6,6) -- (7,6) -- (7,7) -- cycle;
			\fill[color=zzttqq,fill=zzttqq,fill opacity=0.1] (8,7) -- (8,6) -- (9,6) -- (9,7) -- cycle;
			\fill[color=zzttqq,fill=zzttqq,fill opacity=0.1] (1,3) -- (0,3) -- (0,2) -- (1,2) -- cycle;
			\fill[color=zzttqq,fill=zzttqq,fill opacity=0.1] (2,3) -- (2,2) -- (3,2) -- (3,3) -- cycle;
			\fill[color=zzttqq,fill=zzttqq,fill opacity=0.1] (1,1) -- (0,1) -- (0,0) -- (1,0) -- cycle;
			\fill[color=zzttqq,fill=zzttqq,fill opacity=0.1] (2,1) -- (2,0) -- (3,0) -- (3,1) -- cycle;
			\fill[color=zzttqq,fill=zzttqq,fill opacity=0.1] (6,3) -- (7,3) -- (7,2) -- (6,2) -- cycle;
			\fill[color=zzttqq,fill=zzttqq,fill opacity=0.1] (8,3) -- (9,3) -- (9,2) -- (8,2) -- cycle;
			\fill[color=zzttqq,fill=zzttqq,fill opacity=0.1] (6,1) -- (6,0) -- (7,0) -- (7,1) -- cycle;
			\fill[color=zzttqq,fill=zzttqq,fill opacity=0.1] (8,1) -- (8,0) -- (9,0) -- (9,1) -- cycle;
			\fill[color=zzttqq,fill=zzttqq,fill opacity=0.1] (18,3) -- (18,2) -- (19,2) -- (19,3) -- cycle;
			\fill[color=zzttqq,fill=zzttqq,fill opacity=0.1] (20,3) -- (20,2) -- (21,2) -- (21,3) -- cycle;
			\fill[color=zzttqq,fill=zzttqq,fill opacity=0.1] (18,1) -- (18,0) -- (19,0) -- (19,1) -- cycle;
			\fill[color=zzttqq,fill=zzttqq,fill opacity=0.1] (20,1) -- (20,0) -- (21,0) -- (21,1) -- cycle;
			\fill[color=zzttqq,fill=zzttqq,fill opacity=0.1] (24,1) -- (24,0) -- (25,0) -- (25,1) -- cycle;
			\fill[color=zzttqq,fill=zzttqq,fill opacity=0.1] (26,1) -- (26,0) -- (27,0) -- (27,1) -- cycle;
			\fill[color=zzttqq,fill=zzttqq,fill opacity=0.1] (27,2) -- (26,2) -- (26,3) -- (27,3) -- cycle;
			\fill[color=zzttqq,fill=zzttqq,fill opacity=0.1] (25,3) -- (24,3) -- (24,2) -- (25,2) -- cycle;
			\fill[color=zzttqq,fill=zzttqq,fill opacity=0.1] (27,6) -- (26,6) -- (26,7) -- (27,7) -- cycle;
			\fill[color=zzttqq,fill=zzttqq,fill opacity=0.1] (27,8) -- (27,9) -- (26,9) -- (26,8) -- cycle;
			\fill[color=zzttqq,fill=zzttqq,fill opacity=0.1] (25,9) -- (25,8) -- (24,8) -- (24,9) -- cycle;
			\fill[color=zzttqq,fill=zzttqq,fill opacity=0.1] (25,7) -- (24,7) -- (24,6) -- (25,6) -- cycle;
			\fill[color=zzttqq,fill=zzttqq,fill opacity=0.1] (20,9) -- (21,9) -- (21,8) -- (20,8) -- cycle;
			\fill[color=zzttqq,fill=zzttqq,fill opacity=0.1] (19,9) -- (18,9) -- (18,8) -- (19,8) -- cycle;
			\fill[color=zzttqq,fill=zzttqq,fill opacity=0.1] (18,7) -- (19,7) -- (19,6) -- (18,6) -- cycle;
			\fill[color=zzttqq,fill=zzttqq,fill opacity=0.1] (20,7) -- (21,7) -- (21,6) -- (20,6) -- cycle;
			\fill[color=zzttqq,fill=zzttqq,fill opacity=0.1] (0,25) -- (1,25) -- (1,24) -- (0,24) -- cycle;
			\fill[color=zzttqq,fill=zzttqq,fill opacity=0.1] (2,25) -- (3,25) -- (3,24) -- (2,24) -- cycle;
			\fill[color=zzttqq,fill=zzttqq,fill opacity=0.1] (1,27) -- (1,26) -- (0,26) -- (0,27) -- cycle;
			\fill[color=zzttqq,fill=zzttqq,fill opacity=0.1] (2,27) -- (2,26) -- (3,26) -- (3,27) -- cycle;
			\fill[color=zzttqq,fill=zzttqq,fill opacity=0.1] (6,27) -- (6,26) -- (7,26) -- (7,27) -- cycle;
			\fill[color=zzttqq,fill=zzttqq,fill opacity=0.1] (8,27) -- (8,26) -- (9,26) -- (9,27) -- cycle;
			\fill[color=zzttqq,fill=zzttqq,fill opacity=0.1] (6,25) -- (6,24) -- (7,24) -- (7,25) -- cycle;
			\fill[color=zzttqq,fill=zzttqq,fill opacity=0.1] (8,25) -- (8,24) -- (9,24) -- (9,25) -- cycle;
			\fill[color=zzttqq,fill=zzttqq,fill opacity=0.1] (1,21) -- (0,21) -- (0,20) -- (1,20) -- cycle;
			\fill[color=zzttqq,fill=zzttqq,fill opacity=0.1] (2,21) -- (2,20) -- (3,20) -- (3,21) -- cycle;
			\fill[color=zzttqq,fill=zzttqq,fill opacity=0.1] (1,19) -- (0,19) -- (0,18) -- (1,18) -- cycle;
			\fill[color=zzttqq,fill=zzttqq,fill opacity=0.1] (2,19) -- (2,18) -- (3,18) -- (3,19) -- cycle;
			\draw (0,27)-- (27,27);
			\draw (27,27)-- (27,0);
			\draw (27,0)-- (0,0);
			\draw (0,0)-- (0,26);
			\draw (0,26)-- (0,27);
			\draw [color=zzttqq] (6,21)-- (6,20);
			\draw [color=zzttqq] (6,20)-- (7,20);
			\draw [color=zzttqq] (7,20)-- (7,21);
			\draw [color=zzttqq] (7,21)-- (6,21);
			\draw [color=zzttqq] (8,21)-- (9,21);
			\draw [color=zzttqq] (9,21)-- (9,20);
			\draw [color=zzttqq] (9,20)-- (8,20);
			\draw [color=zzttqq] (8,20)-- (8,21);
			\draw [color=zzttqq] (7,19)-- (7,18);
			\draw [color=zzttqq] (7,18)-- (6,18);
			\draw [color=zzttqq] (6,18)-- (6,19);
			\draw [color=zzttqq] (6,19)-- (7,19);
			\draw [color=zzttqq] (18,19)-- (19,19);
			\draw [color=zzttqq] (19,19)-- (19,18);
			\draw [color=zzttqq] (19,18)-- (18,18);
			\draw [color=zzttqq] (18,18)-- (18,19);
			\draw [color=zzttqq] (21,21)-- (20,21);
			\draw [color=zzttqq] (20,21)-- (20,20);
			\draw [color=zzttqq] (20,20)-- (21,20);
			\draw [color=zzttqq] (21,20)-- (21,21);
			\draw [color=zzttqq] (24,21)-- (24,20);
			\draw [color=zzttqq] (24,20)-- (25,20);
			\draw [color=zzttqq] (25,20)-- (25,21);
			\draw [color=zzttqq] (25,21)-- (24,21);
			\draw [color=zzttqq] (26,21)-- (26,20);
			\draw [color=zzttqq] (26,20)-- (27,20);
			\draw [color=zzttqq] (27,20)-- (27,21);
			\draw [color=zzttqq] (27,21)-- (26,21);
			\draw [color=zzttqq] (24,19)-- (24,18);
			\draw [color=zzttqq] (24,18)-- (25,18);
			\draw [color=zzttqq] (25,18)-- (25,19);
			\draw [color=zzttqq] (25,19)-- (24,19);
			\draw [color=zzttqq] (26,19)-- (26,18);
			\draw [color=zzttqq] (26,18)-- (27,18);
			\draw [color=zzttqq] (27,18)-- (27,19);
			\draw [color=zzttqq] (27,19)-- (26,19);
			\draw [color=zzttqq] (24,27)-- (24,26);
			\draw [color=zzttqq] (24,26)-- (25,26);
			\draw [color=zzttqq] (25,26)-- (25,27);
			\draw [color=zzttqq] (25,27)-- (24,27);
			\draw [color=zzttqq] (26,27)-- (26,26);
			\draw [color=zzttqq] (26,26)-- (27,26);
			\draw [color=zzttqq] (27,26)-- (27,27);
			\draw [color=zzttqq] (27,27)-- (26,27);
			\draw [color=zzttqq] (24,25)-- (24,24);
			\draw [color=zzttqq] (24,24)-- (25,24);
			\draw [color=zzttqq] (25,24)-- (25,25);
			\draw [color=zzttqq] (25,25)-- (24,25);
			\draw [color=zzttqq] (26,25)-- (26,24);
			\draw [color=zzttqq] (26,24)-- (27,24);
			\draw [color=zzttqq] (27,24)-- (27,25);
			\draw [color=zzttqq] (27,25)-- (26,25);
			\draw [color=zzttqq] (21,27)-- (21,26);
			\draw [color=zzttqq] (21,26)-- (20,26);
			\draw [color=zzttqq] (20,26)-- (20,27);
			\draw [color=zzttqq] (20,27)-- (21,27);
			\draw [color=zzttqq] (19,26)-- (18,26);
			\draw [color=zzttqq] (18,26)-- (18,27);
			\draw [color=zzttqq] (18,27)-- (19,27);
			\draw [color=zzttqq] (19,27)-- (19,26);
			\draw [color=zzttqq] (18,25)-- (18,24);
			\draw [color=zzttqq] (18,24)-- (19,24);
			\draw [color=zzttqq] (19,24)-- (19,25);
			\draw [color=zzttqq] (19,25)-- (18,25);
			\draw [color=zzttqq] (20,25)-- (20,24);
			\draw [color=zzttqq] (20,24)-- (21,24);
			\draw [color=zzttqq] (21,24)-- (21,25);
			\draw [color=zzttqq] (21,25)-- (20,25);
			\draw [color=zzttqq] (8,19)-- (8,18);
			\draw [color=zzttqq] (8,18)-- (9,18);
			\draw [color=zzttqq] (9,18)-- (9,19);
			\draw [color=zzttqq] (9,19)-- (8,19);
			\draw [color=zzttqq] (19,21)-- (18,21);
			\draw [color=zzttqq] (18,21)-- (18,20);
			\draw [color=zzttqq] (18,20)-- (19,20);
			\draw [color=zzttqq] (19,20)-- (19,21);
			\draw [color=zzttqq] (20,19)-- (20,18);
			\draw [color=zzttqq] (20,18)-- (21,18);
			\draw [color=zzttqq] (21,18)-- (21,19);
			\draw [color=zzttqq] (21,19)-- (20,19);
			\draw [color=zzttqq] (1,9)-- (0,9);
			\draw [color=zzttqq] (0,9)-- (0,8);
			\draw [color=zzttqq] (0,8)-- (1,8);
			\draw [color=zzttqq] (1,8)-- (1,9);
			\draw [color=zzttqq] (1,7)-- (0,7);
			\draw [color=zzttqq] (0,7)-- (0,6);
			\draw [color=zzttqq] (0,6)-- (1,6);
			\draw [color=zzttqq] (1,6)-- (1,7);
			\draw [color=zzttqq] (2,9)-- (2,8);
			\draw [color=zzttqq] (2,8)-- (3,8);
			\draw [color=zzttqq] (3,8)-- (3,9);
			\draw [color=zzttqq] (3,9)-- (2,9);
			\draw [color=zzttqq] (2,7)-- (2,6);
			\draw [color=zzttqq] (2,6)-- (3,6);
			\draw [color=zzttqq] (3,6)-- (3,7);
			\draw [color=zzttqq] (3,7)-- (2,7);
			\draw [color=zzttqq] (6,9)-- (6,8);
			\draw [color=zzttqq] (6,8)-- (7,8);
			\draw [color=zzttqq] (7,8)-- (7,9);
			\draw [color=zzttqq] (7,9)-- (6,9);
			\draw [color=zzttqq] (8,9)-- (8,8);
			\draw [color=zzttqq] (8,8)-- (9,8);
			\draw [color=zzttqq] (9,8)-- (9,9);
			\draw [color=zzttqq] (9,9)-- (8,9);
			\draw [color=zzttqq] (6,7)-- (6,6);
			\draw [color=zzttqq] (6,6)-- (7,6);
			\draw [color=zzttqq] (7,6)-- (7,7);
			\draw [color=zzttqq] (7,7)-- (6,7);
			\draw [color=zzttqq] (8,7)-- (8,6);
			\draw [color=zzttqq] (8,6)-- (9,6);
			\draw [color=zzttqq] (9,6)-- (9,7);
			\draw [color=zzttqq] (9,7)-- (8,7);
			\draw [color=zzttqq] (1,3)-- (0,3);
			\draw [color=zzttqq] (0,3)-- (0,2);
			\draw [color=zzttqq] (0,2)-- (1,2);
			\draw [color=zzttqq] (1,2)-- (1,3);
			\draw [color=zzttqq] (2,3)-- (2,2);
			\draw [color=zzttqq] (2,2)-- (3,2);
			\draw [color=zzttqq] (3,2)-- (3,3);
			\draw [color=zzttqq] (3,3)-- (2,3);
			\draw [color=zzttqq] (1,1)-- (0,1);
			\draw [color=zzttqq] (0,1)-- (0,0);
			\draw [color=zzttqq] (0,0)-- (1,0);
			\draw [color=zzttqq] (1,0)-- (1,1);
			\draw [color=zzttqq] (2,1)-- (2,0);
			\draw [color=zzttqq] (2,0)-- (3,0);
			\draw [color=zzttqq] (3,0)-- (3,1);
			\draw [color=zzttqq] (3,1)-- (2,1);
			\draw [color=zzttqq] (6,3)-- (7,3);
			\draw [color=zzttqq] (7,3)-- (7,2);
			\draw [color=zzttqq] (7,2)-- (6,2);
			\draw [color=zzttqq] (6,2)-- (6,3);
			\draw [color=zzttqq] (8,3)-- (9,3);
			\draw [color=zzttqq] (9,3)-- (9,2);
			\draw [color=zzttqq] (9,2)-- (8,2);
			\draw [color=zzttqq] (8,2)-- (8,3);
			\draw [color=zzttqq] (6,1)-- (6,0);
			\draw [color=zzttqq] (6,0)-- (7,0);
			\draw [color=zzttqq] (7,0)-- (7,1);
			\draw [color=zzttqq] (7,1)-- (6,1);
			\draw [color=zzttqq] (8,1)-- (8,0);
			\draw [color=zzttqq] (8,0)-- (9,0);
			\draw [color=zzttqq] (9,0)-- (9,1);
			\draw [color=zzttqq] (9,1)-- (8,1);
			\draw [color=zzttqq] (18,3)-- (18,2);
			\draw [color=zzttqq] (18,2)-- (19,2);
			\draw [color=zzttqq] (19,2)-- (19,3);
			\draw [color=zzttqq] (19,3)-- (18,3);
			\draw [color=zzttqq] (20,3)-- (20,2);
			\draw [color=zzttqq] (20,2)-- (21,2);
			\draw [color=zzttqq] (21,2)-- (21,3);
			\draw [color=zzttqq] (21,3)-- (20,3);
			\draw [color=zzttqq] (18,1)-- (18,0);
			\draw [color=zzttqq] (18,0)-- (19,0);
			\draw [color=zzttqq] (19,0)-- (19,1);
			\draw [color=zzttqq] (19,1)-- (18,1);
			\draw [color=zzttqq] (20,1)-- (20,0);
			\draw [color=zzttqq] (20,0)-- (21,0);
			\draw [color=zzttqq] (21,0)-- (21,1);
			\draw [color=zzttqq] (21,1)-- (20,1);
			\draw [color=zzttqq] (24,1)-- (24,0);
			\draw [color=zzttqq] (24,0)-- (25,0);
			\draw [color=zzttqq] (25,0)-- (25,1);
			\draw [color=zzttqq] (25,1)-- (24,1);
			\draw [color=zzttqq] (26,1)-- (26,0);
			\draw [color=zzttqq] (26,0)-- (27,0);
			\draw [color=zzttqq] (27,0)-- (27,1);
			\draw [color=zzttqq] (27,1)-- (26,1);
			\draw [color=zzttqq] (27,2)-- (26,2);
			\draw [color=zzttqq] (26,2)-- (26,3);
			\draw [color=zzttqq] (26,3)-- (27,3);
			\draw [color=zzttqq] (27,3)-- (27,2);
			\draw [color=zzttqq] (25,3)-- (24,3);
			\draw [color=zzttqq] (24,3)-- (24,2);
			\draw [color=zzttqq] (24,2)-- (25,2);
			\draw [color=zzttqq] (25,2)-- (25,3);
			\draw [color=zzttqq] (27,6)-- (26,6);
			\draw [color=zzttqq] (26,6)-- (26,7);
			\draw [color=zzttqq] (26,7)-- (27,7);
			\draw [color=zzttqq] (27,7)-- (27,6);
			\draw [color=zzttqq] (27,8)-- (27,9);
			\draw [color=zzttqq] (27,9)-- (26,9);
			\draw [color=zzttqq] (26,9)-- (26,8);
			\draw [color=zzttqq] (26,8)-- (27,8);
			\draw [color=zzttqq] (25,9)-- (25,8);
			\draw [color=zzttqq] (25,8)-- (24,8);
			\draw [color=zzttqq] (24,8)-- (24,9);
			\draw [color=zzttqq] (24,9)-- (25,9);
			\draw [color=zzttqq] (25,7)-- (24,7);
			\draw [color=zzttqq] (24,7)-- (24,6);
			\draw [color=zzttqq] (24,6)-- (25,6);
			\draw [color=zzttqq] (25,6)-- (25,7);
			\draw [color=zzttqq] (20,9)-- (21,9);
			\draw [color=zzttqq] (21,9)-- (21,8);
			\draw [color=zzttqq] (21,8)-- (20,8);
			\draw [color=zzttqq] (20,8)-- (20,9);
			\draw [color=zzttqq] (19,9)-- (18,9);
			\draw [color=zzttqq] (18,9)-- (18,8);
			\draw [color=zzttqq] (18,8)-- (19,8);
			\draw [color=zzttqq] (19,8)-- (19,9);
			\draw [color=zzttqq] (18,7)-- (19,7);
			\draw [color=zzttqq] (19,7)-- (19,6);
			\draw [color=zzttqq] (19,6)-- (18,6);
			\draw [color=zzttqq] (18,6)-- (18,7);
			\draw [color=zzttqq] (20,7)-- (21,7);
			\draw [color=zzttqq] (21,7)-- (21,6);
			\draw [color=zzttqq] (21,6)-- (20,6);
			\draw [color=zzttqq] (20,6)-- (20,7);
			\draw [color=zzttqq] (0,25)-- (1,25);
			\draw [color=zzttqq] (1,25)-- (1,24);
			\draw [color=zzttqq] (1,24)-- (0,24);
			\draw [color=zzttqq] (0,24)-- (0,25);
			\draw [color=zzttqq] (2,25)-- (3,25);
			\draw [color=zzttqq] (3,25)-- (3,24);
			\draw [color=zzttqq] (3,24)-- (2,24);
			\draw [color=zzttqq] (2,24)-- (2,25);
			\draw [color=zzttqq] (1,27)-- (1,26);
			\draw [color=zzttqq] (1,26)-- (0,26);
			\draw [color=zzttqq] (0,26)-- (0,27);
			\draw [color=zzttqq] (0,27)-- (1,27);
			\draw [color=zzttqq] (2,27)-- (2,26);
			\draw [color=zzttqq] (2,26)-- (3,26);
			\draw [color=zzttqq] (3,26)-- (3,27);
			\draw [color=zzttqq] (3,27)-- (2,27);
			\draw [color=zzttqq] (6,27)-- (6,26);
			\draw [color=zzttqq] (6,26)-- (7,26);
			\draw [color=zzttqq] (7,26)-- (7,27);
			\draw [color=zzttqq] (7,27)-- (6,27);
			\draw [color=zzttqq] (8,27)-- (8,26);
			\draw [color=zzttqq] (8,26)-- (9,26);
			\draw [color=zzttqq] (9,26)-- (9,27);
			\draw [color=zzttqq] (9,27)-- (8,27);
			\draw [color=zzttqq] (6,25)-- (6,24);
			\draw [color=zzttqq] (6,24)-- (7,24);
			\draw [color=zzttqq] (7,24)-- (7,25);
			\draw [color=zzttqq] (7,25)-- (6,25);
			\draw [color=zzttqq] (8,25)-- (8,24);
			\draw [color=zzttqq] (8,24)-- (9,24);
			\draw [color=zzttqq] (9,24)-- (9,25);
			\draw [color=zzttqq] (9,25)-- (8,25);
			\draw [color=zzttqq] (1,21)-- (0,21);
			\draw [color=zzttqq] (0,21)-- (0,20);
			\draw [color=zzttqq] (0,20)-- (1,20);
			\draw [color=zzttqq] (1,20)-- (1,21);
			\draw [color=zzttqq] (2,21)-- (2,20);
			\draw [color=zzttqq] (2,20)-- (3,20);
			\draw [color=zzttqq] (3,20)-- (3,21);
			\draw [color=zzttqq] (3,21)-- (2,21);
			\draw [color=zzttqq] (1,19)-- (0,19);
			\draw [color=zzttqq] (0,19)-- (0,18);
			\draw [color=zzttqq] (0,18)-- (1,18);
			\draw [color=zzttqq] (1,18)-- (1,19);
			\draw [color=zzttqq] (2,19)-- (2,18);
			\draw [color=zzttqq] (2,18)-- (3,18);
			\draw [color=zzttqq] (3,18)-- (3,19);
			\draw [color=zzttqq] (3,19)-- (2,19);
			\draw (7.7,16.38)-- (7,18);
			\draw(7.7,16.38) circle (0.62cm);
			\draw (3.95,13.08) node[anchor=north west] {{\tiny $B(x,ad(x,\Omega^c))$}};
			\draw (8.16,16.89) node[anchor=north west] {{\tiny x}};
			\draw (3.72,18.43) node[anchor=north west] {{\tiny $Q'$}};
			\begin{scriptsize}
				\fill [color=black] (0,27) circle (1.0pt);
				\fill [color=black] (0,9) circle (1.0pt);
				\fill [color=black] (0,18) circle (1.0pt);
				\fill [color=black] (9,0) circle (1.0pt);
				\fill [color=black] (18,0) circle (1.0pt);
				\fill [color=black] (27,0) circle (1.0pt);
				\fill [color=black] (21,0) circle (1.0pt);
				\fill [color=black] (24,0) circle (1.0pt);
				\fill [color=black] (25,0) circle (1.0pt);
				\fill [color=black] (26,0) circle (1.0pt);
				\fill [color=black] (19,0) circle (1.0pt);
				\fill [color=black] (20,0) circle (1.0pt);
				\fill [color=black] (3,0) circle (1.0pt);
				\fill [color=black] (6,0) circle (1.0pt);
				\fill [color=black] (1,0) circle (1.0pt);
				\fill [color=black] (2,0) circle (1.0pt);
				\fill [color=black] (7,0) circle (1.0pt);
				\fill [color=black] (8,0) circle (1.0pt);
				\fill [color=black] (0,1) circle (1.0pt);
				\fill [color=uququq] (0,0) circle (1.0pt);
				\fill [color=black] (0,2) circle (1.0pt);
				\fill [color=black] (0,3) circle (1.0pt);
				\fill [color=black] (0,6) circle (1.0pt);
				\fill [color=black] (0,7) circle (1.0pt);
				\fill [color=black] (0,8) circle (1.0pt);
				\fill [color=black] (0,21) circle (1.0pt);
				\fill [color=black] (0,24) circle (1.0pt);
				\fill [color=black] (0,19) circle (1.0pt);
				\fill [color=black] (0,20) circle (1.0pt);
				\fill [color=black] (0,25) circle (1.0pt);
				\fill [color=black] (0,26) circle (1.0pt);
				\fill [color=uququq] (27,27) circle (1.0pt);
				\fill [color=uququq] (9,27) circle (1.0pt);
				\fill [color=uququq] (18,27) circle (1.0pt);
				\fill [color=uququq] (18,18) circle (1.0pt);
				\fill [color=uququq] (9,18) circle (1.0pt);
				\fill [color=uququq] (9,9) circle (1.0pt);
				\fill [color=uququq] (18,9) circle (1.0pt);
				\fill [color=uququq] (27,18) circle (1.0pt);
				\fill [color=uququq] (27,9) circle (1.0pt);
				\fill [color=uququq] (3,27) circle (1.0pt);
				\fill [color=uququq] (3,24) circle (1.0pt);
				\fill [color=uququq] (6,24) circle (1.0pt);
				\fill [color=uququq] (6,27) circle (1.0pt);
				\fill [color=uququq] (6,21) circle (1.0pt);
				\fill [color=uququq] (3,21) circle (1.0pt);
				\fill [color=uququq] (3,18) circle (1.0pt);
				\fill [color=uququq] (6,18) circle (1.0pt);
				\fill [color=uququq] (9,21) circle (1.0pt);
				\fill [color=uququq] (9,24) circle (1.0pt);
				\fill [color=uququq] (18,24) circle (1.0pt);
				\fill [color=uququq] (18,21) circle (1.0pt);
				\fill [color=uququq] (21,24) circle (1.0pt);
				\fill [color=uququq] (21,27) circle (1.0pt);
				\fill [color=uququq] (24,27) circle (1.0pt);
				\fill [color=uququq] (24,24) circle (1.0pt);
				\fill [color=uququq] (27,24) circle (1.0pt);
				\fill [color=uququq] (27,21) circle (1.0pt);
				\fill [color=uququq] (24,21) circle (1.0pt);
				\fill [color=uququq] (21,21) circle (1.0pt);
				\fill [color=uququq] (21,18) circle (1.0pt);
				\fill [color=uququq] (24,18) circle (1.0pt);
				\fill [color=uququq] (21,9) circle (1.0pt);
				\fill [color=uququq] (24,9) circle (1.0pt);
				\fill [color=uququq] (24,6) circle (1.0pt);
				\fill [color=uququq] (21,6) circle (1.0pt);
				\fill [color=uququq] (21,3) circle (1.0pt);
				\fill [color=uququq] (24,3) circle (1.0pt);
				\fill [color=uququq] (27,3) circle (1.0pt);
				\fill [color=uququq] (27,6) circle (1.0pt);
				\fill [color=uququq] (18,3) circle (1.0pt);
				\fill [color=uququq] (9,3) circle (1.0pt);
				\fill [color=uququq] (9,6) circle (1.0pt);
				\fill [color=uququq] (6,6) circle (1.0pt);
				\fill [color=uququq] (6,9) circle (1.0pt);
				\fill [color=uququq] (3,9) circle (1.0pt);
				\fill [color=uququq] (3,6) circle (1.0pt);
				\fill [color=uququq] (3,3) circle (1.0pt);
				\fill [color=uququq] (6,3) circle (1.0pt);
				\fill [color=uququq] (18,6) circle (1.0pt);
				\fill [color=uququq] (3,26) circle (1.0pt);
				\fill [color=uququq] (3,25) circle (1.0pt);
				\fill [color=uququq] (6,26) circle (1.0pt);
				\fill [color=uququq] (6,25) circle (1.0pt);
				\fill [color=uququq] (9,26) circle (1.0pt);
				\fill [color=uququq] (21,25) circle (1.0pt);
				\fill [color=uququq] (21,26) circle (1.0pt);
				\fill [color=uququq] (24,26) circle (1.0pt);
				\fill [color=uququq] (24,25) circle (1.0pt);
				\fill [color=uququq] (27,26) circle (1.0pt);
				\fill [color=uququq] (27,25) circle (1.0pt);
				\fill [color=uququq] (18,25) circle (1.0pt);
				\fill [color=uququq] (9,25) circle (1.0pt);
				\fill [color=uququq] (3,20) circle (1.0pt);
				\fill [color=uququq] (3,19) circle (1.0pt);
				\fill [color=uququq] (6,20) circle (1.0pt);
				\fill [color=uququq] (6,19) circle (1.0pt);
				\fill [color=uququq] (9,19) circle (1.0pt);
				\fill [color=uququq] (9,20) circle (1.0pt);
				\fill [color=uququq] (18,20) circle (1.0pt);
				\fill [color=uququq] (18,19) circle (1.0pt);
				\fill [color=uququq] (21,19) circle (1.0pt);
				\fill [color=uququq] (21,20) circle (1.0pt);
				\fill [color=uququq] (24,20) circle (1.0pt);
				\fill [color=uququq] (24,19) circle (1.0pt);
				\fill [color=uququq] (27,19) circle (1.0pt);
				\fill [color=uququq] (27,20) circle (1.0pt);
				\fill [color=uququq] (3,8) circle (1.0pt);
				\fill [color=uququq] (3,7) circle (1.0pt);
				\fill [color=uququq] (3,2) circle (1.0pt);
				\fill [color=uququq] (3,1) circle (1.0pt);
				\fill [color=black] (6,2) circle (1.0pt);
				\fill [color=uququq] (6,1) circle (1.0pt);
				\fill [color=uququq] (9,2) circle (1.0pt);
				\fill [color=uququq] (9,1) circle (1.0pt);
				\fill [color=uququq] (9,8) circle (1.0pt);
				\fill [color=uququq] (18,8) circle (1.0pt);
				\fill [color=uququq] (24,8) circle (1.0pt);
				\fill [color=uququq] (24,7) circle (1.0pt);
				\fill [color=uququq] (21,7) circle (1.0pt);
				\fill [color=uququq] (9,7) circle (1.0pt);
				\fill [color=uququq] (18,7) circle (1.0pt);
				\fill [color=uququq] (18,2) circle (1.0pt);
				\fill [color=uququq] (18,1) circle (1.0pt);
				\fill [color=uququq] (24,2) circle (1.0pt);
				\fill [color=uququq] (24,1) circle (1.0pt);
				\fill [color=uququq] (21,1) circle (1.0pt);
				\fill [color=uququq] (21,2) circle (1.0pt);
				\fill [color=uququq] (27,2) circle (1.0pt);
				\fill [color=uququq] (27,1) circle (1.0pt);
				\fill [color=uququq] (1,8) circle (1.0pt);
				\fill [color=uququq] (1,9) circle (1.0pt);
				\fill [color=uququq] (2,9) circle (1.0pt);
				\fill [color=uququq] (2,8) circle (1.0pt);
				\fill [color=uququq] (2,7) circle (1.0pt);
				\fill [color=uququq] (1,7) circle (1.0pt);
				\fill [color=uququq] (1,6) circle (1.0pt);
				\fill [color=uququq] (2,6) circle (1.0pt);
				\fill [color=uququq] (7,6) circle (1.0pt);
				\fill [color=uququq] (8,6) circle (1.0pt);
				\fill [color=uququq] (8,7) circle (1.0pt);
				\fill [color=uququq] (8,8) circle (1.0pt);
				\fill [color=uququq] (7,8) circle (1.0pt);
				\fill [color=black] (6,8) circle (1.0pt);
				\fill [color=uququq] (6,7) circle (1.0pt);
				\fill [color=uququq] (7,7) circle (1.0pt);
				\fill [color=uququq] (19,9) circle (1.0pt);
				\fill [color=uququq] (20,9) circle (1.0pt);
				\fill [color=uququq] (20,8) circle (1.0pt);
				\fill [color=uququq] (19,8) circle (1.0pt);
				\fill [color=uququq] (19,7) circle (1.0pt);
				\fill [color=uququq] (20,7) circle (1.0pt);
				\fill [color=uququq] (19,6) circle (1.0pt);
				\fill [color=uququq] (20,6) circle (1.0pt);
				\fill [color=uququq] (25,8) circle (1.0pt);
				\fill [color=uququq] (26,8) circle (1.0pt);
				\fill [color=uququq] (26,9) circle (1.0pt);
				\fill [color=uququq] (25,9) circle (1.0pt);
				\fill [color=uququq] (26,7) circle (1.0pt);
				\fill [color=uququq] (25,7) circle (1.0pt);
				\fill [color=uququq] (27,8) circle (1.0pt);
				\fill [color=uququq] (27,7) circle (1.0pt);
				\fill [color=uququq] (26,6) circle (1.0pt);
				\fill [color=uququq] (25,6) circle (1.0pt);
				\fill [color=uququq] (20,3) circle (1.0pt);
				\fill [color=uququq] (19,3) circle (1.0pt);
				\fill [color=uququq] (19,2) circle (1.0pt);
				\fill [color=uququq] (20,2) circle (1.0pt);
				\fill [color=uququq] (19,1) circle (1.0pt);
				\fill [color=uququq] (20,1) circle (1.0pt);
				\fill [color=uququq] (25,3) circle (1.0pt);
				\fill [color=uququq] (26,3) circle (1.0pt);
				\fill [color=uququq] (26,2) circle (1.0pt);
				\fill [color=uququq] (25,2) circle (1.0pt);
				\fill [color=uququq] (25,1) circle (1.0pt);
				\fill [color=uququq] (26,1) circle (1.0pt);
				\fill [color=uququq] (8,1) circle (1.0pt);
				\fill [color=uququq] (7,1) circle (1.0pt);
				\fill [color=uququq] (7,2) circle (1.0pt);
				\fill [color=uququq] (8,2) circle (1.0pt);
				\fill [color=uququq] (8,3) circle (1.0pt);
				\fill [color=uququq] (7,3) circle (1.0pt);
				\fill [color=uququq] (2,3) circle (1.0pt);
				\fill [color=uququq] (2,2) circle (1.0pt);
				\fill [color=uququq] (1,2) circle (1.0pt);
				\fill [color=uququq] (1,3) circle (1.0pt);
				\fill [color=uququq] (1,1) circle (1.0pt);
				\fill [color=uququq] (2,1) circle (1.0pt);
				\fill [color=uququq] (8,9) circle (1.0pt);
				\fill [color=uququq] (7,9) circle (1.0pt);
				\fill [color=uququq] (21,8) circle (1.0pt);
				\fill [color=uququq] (1,27) circle (1.0pt);
				\fill [color=uququq] (2,27) circle (1.0pt);
				\fill [color=uququq] (2,26) circle (1.0pt);
				\fill [color=uququq] (1,26) circle (1.0pt);
				\fill [color=uququq] (1,25) circle (1.0pt);
				\fill [color=uququq] (2,25) circle (1.0pt);
				\fill [color=uququq] (1,24) circle (1.0pt);
				\fill [color=uququq] (2,24) circle (1.0pt);
				\fill [color=uququq] (7,24) circle (1.0pt);
				\fill [color=uququq] (7,25) circle (1.0pt);
				\fill [color=uququq] (7,26) circle (1.0pt);
				\fill [color=uququq] (7,27) circle (1.0pt);
				\fill [color=uququq] (8,27) circle (1.0pt);
				\fill [color=uququq] (8,26) circle (1.0pt);
				\fill [color=uququq] (8,25) circle (1.0pt);
				\fill [color=uququq] (8,24) circle (1.0pt);
				\fill [color=uququq] (2,21) circle (1.0pt);
				\fill [color=uququq] (1,21) circle (1.0pt);
				\fill [color=uququq] (1,20) circle (1.0pt);
				\fill [color=uququq] (1,19) circle (1.0pt);
				\fill [color=uququq] (1,18) circle (1.0pt);
				\fill [color=uququq] (2,18) circle (1.0pt);
				\fill [color=uququq] (2,19) circle (1.0pt);
				\fill [color=uququq] (2,20) circle (1.0pt);
				\fill [color=uququq] (7,21) circle (1.0pt);
				\fill [color=uququq] (8,21) circle (1.0pt);
				\fill [color=uququq] (7,20) circle (1.0pt);
				\fill [color=uququq] (8,20) circle (1.0pt);
				\fill [color=uququq] (8,19) circle (1.0pt);
				\fill [color=uququq] (7,19) circle (1.0pt);
				\fill [color=uququq] (7,18) circle (1.0pt);
				\fill [color=uququq] (8,18) circle (1.0pt);
				\fill [color=uququq] (19,21) circle (1.0pt);
				\fill [color=uququq] (20,20) circle (1.0pt);
				\fill [color=uququq] (19,19) circle (1.0pt);
				\fill [color=uququq] (19,18) circle (1.0pt);
				\fill [color=uququq] (20,18) circle (1.0pt);
				\fill [color=uququq] (20,21) circle (1.0pt);
				\fill [color=uququq] (25,21) circle (1.0pt);
				\fill [color=uququq] (26,20) circle (1.0pt);
				\fill [color=uququq] (25,20) circle (1.0pt);
				\fill [color=uququq] (25,19) circle (1.0pt);
				\fill [color=uququq] (26,19) circle (1.0pt);
				\fill [color=uququq] (25,18) circle (1.0pt);
				\fill [color=uququq] (26,18) circle (1.0pt);
				\fill [color=uququq] (26,24) circle (1.0pt);
				\fill [color=uququq] (25,24) circle (1.0pt);
				\fill [color=uququq] (25,25) circle (1.0pt);
				\fill [color=uququq] (26,25) circle (1.0pt);
				\fill [color=uququq] (26,26) circle (1.0pt);
				\fill [color=uququq] (25,26) circle (1.0pt);
				\fill [color=uququq] (25,27) circle (1.0pt);
				\fill [color=uququq] (26,27) circle (1.0pt);
				\fill [color=uququq] (20,27) circle (1.0pt);
				\fill [color=uququq] (19,27) circle (1.0pt);
				\fill [color=uququq] (19,26) circle (1.0pt);
				\fill [color=uququq] (20,26) circle (1.0pt);
				\fill [color=uququq] (20,25) circle (1.0pt);
				\fill [color=uququq] (19,25) circle (1.0pt);
				\fill [color=uququq] (19,24) circle (1.0pt);
				\fill [color=uququq] (20,24) circle (1.0pt);
				\fill [color=uququq] (26,21) circle (1.0pt);
				\fill [color=uququq] (18,26) circle (1.0pt);
				\fill [color=uququq] (19,20) circle (1.0pt);
				\fill [color=uququq] (20,19) circle (1.0pt);
				\fill [color=black] (7.7,16.38) circle (1.5pt);
			\end{scriptsize}
		\end{tikzpicture}
		\vspace{-.7cm}
		\caption{The white area is in $\Omega\cap Q$ and black dots belong to $\Omega^c$}
	\end{figure}
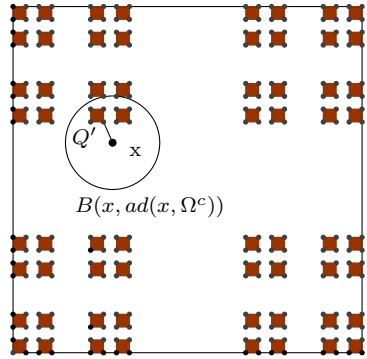

	The condition \eqref{key} in Theorem \ref{theorem1} is equivalent to the geometric condition that characterizes ``sets preserving the global Markov inequality''. A closed set $F\subset\rn$ is said to preserve the global Markov inequality if for every positive integer $k$ and every ball $B_r(y)$ with $y\in F, r>0$, we have
	\begin{equation}\label{setspreservingmarkovmarkov}
		\linfno{\nabla P}{F\cap B_r(y)}\le c(F,k)r^{-1}\linfno{P}{F\cap B_r(y)}, \quad P\in \pk,
	\end{equation}
	where $c(F,k)$ is a constant depending only on $F$ and $k$. It turns out that such sets can be characterized in terms of the width. More specifically, $F$ preserves the global Markov inequality if and only if there exists a constant $\varepsilon>0$ such that for all $y\in F, r>0$ we have
	\begin{equation}\label{globalmarkov}
		\frac{w(F\cap B_r(y))}{r}>\varepsilon
	\end{equation}
	(see \cite{MR0820626} Theorem 2 p. 38, and \cite{MR2868143} Chapter 9 Section 9.2). In the next proposition we will show that $\Omega$ satisfies \eqref{key} exactly when $\Omega^c$ preserves the global Markov inequality.
	
	\begin{proposition}\label{proposition}
		A proper open set $\Omega\subset\rn$ satisfies condition \eqref{key} if and only if $\Omega^c$ preserves the global Markov inequality.
	\end{proposition}
	\begin{proof}
		First suppose $\Omega^c$ preserves the global Markov inequality hence for all $y\in\Omega^c$ and $r>0$, \eqref{globalmarkov} holds with $F=\Omega^c$ and some $\varepsilon>0$. Now for an arbitrary
		$x\in\Omega$ we consider $B(x,2d(x))$ and choose $y\in\Omega^c\cap B(x,2d(x))$ such that $|y-x|=d(x)$. Then we note that $B(y,d(x))\subset B(x,2d(x))$ and this together with \eqref{globalmarkov} gives us
		\[
		w(B(x,2d(x))\cap \Omega^c)\ge w(B(y,d(x))\cap \Omega^c)>\varepsilon d(x),
		\] 
		which proves that \eqref{key} holds for $\Omega$ with $a=2$ and $\delta=\varepsilon$.\\
		
		To prove the other direction, suppose \eqref{key} holds for $\Omega$ with some $a>1$ and $\delta>0$. Now for an arbitrary $y\in \Omega^c$ and $r>0$ we consider the sphere 
		\[
		S(y,\frac{r}{1+a})=\left\{z\in\rn: |z-y|=\frac{r}{1+a}\right\}.
		\]
		Then we note that if $S(y,\frac{r}{1+a})\subset \Omega^c$ we have
		\[
		w(B(y,r)\cap \Omega^c)\ge w(S(y,\frac{r}{1+a}))=\frac{2r}{1+a},
		\]
		which shows that \eqref{globalmarkov} holds with $\varepsilon=\frac{2}{1+a}$. So it remains to assume $S(y,\frac{r}{1+a})\cap \Omega\ne\emptyset$, in which case we set 
		\[
		\varepsilon_0=\sup\left\{d(z): z\in \Omega\cap S(y,\frac{r}{1+a})\right\},
		\]
		and consider two cases $\varepsilon_0<\frac{r}{2(1+a)}$ and $\varepsilon_0\ge\frac{r}{2(1+a)}$. In the first case, for each $z\in S(y,\frac{r}{1+a})$ there exists some $y_z\in B(y,r)\cap \Omega^c$ such that $|y_z-z|<\frac{r}{2(1+a)}$. Therefore, in the direction $\nu=\frac{z-y}{|z-y|}$ we have 
		\[
		|(y-y_z)\cdot\nu|\ge |y-z|-|y_z-z|>\frac{r}{2(1+a)}.
		\]
		This shows that $w(\Omega^c\cap B(y,r))\ge\frac{r}{2(1+a)}$ and hence \eqref{globalmarkov} holds with $\varepsilon=\frac{1}{2(1+a)}$. Finally, when $\varepsilon_0\ge\frac{r}{2(1+a)}$ there is $x\in S(y,\frac{r}{1+a})\cap \Omega$ such that $\frac{r}{2(1+a)}\le d(x)\le\frac{r}{1+a}$. Noting that $B(x,ad(x))\subset B(y,r)$ together with \eqref{key} gives us
		\[
		w(\Omega^c\cap B(y,r))\ge w(\Omega^c\cap B(x,ad(x)))>\delta d(x)\ge \delta\frac{r}{2(1+a)},
		\]
		which shows that \eqref{globalmarkov} holds with $\varepsilon=\frac{\delta}{2(1+a)}$. So in any case the condition \eqref{globalmarkov} holds with $\varepsilon=\frac{\min(1,\delta)}{2(1+a)}$, and this completes the proof.

	\end{proof}
	
	We break the proof of Theorem \ref{theorem1} into two parts, the sufficiency and the necessity. In the next section, we will see it is easy to show that \eqref{h1cc} is enough for extending each $(1,\Omega)$-atom to some 1-atom in $\rn$. To prove that \eqref{key} is sufficient, we need to show that it allows us to choose an interpolation problem and solve it properly, which is the content of Lemma \ref{interpolationlemma}. This minor difficulty comes from the fact that when $n(\frac{1}{p}-1) \in \mathbb{N}$, the order of distributions in $\hp$ can not exceed than $n(\frac{1}{p}-1)-1$, while they must have vanishing moments up to order $n(\frac{1}{p}-1)$. This becomes more clear when it is compared to the case $n(\frac{1}{p}-1)\notin \mathbb{N}\cup \{0\}$.
	\section{Proof of sufficiency}
	We begin with the following lemma, which looks like a reverse Markov's inequality \eqref{setspreservingmarkovmarkov} and is already known. Nevertheless, for the convenience of the reader we give a proof of this here (see \cite{MR0820626} Remark 1 p. 35,  and \cite{MR2868143} Theorem 9.21).  
	\begin{lemma}\label{ckhyperlemma}
		Suppose $E$ is a subset of the unit ball $B$ of $\rn$. Then for any $k=1,2,..$, we have
		
		\begin{equation}\label{hpc}
			{\|P\|}_{C^{k-1}(E)} > c(n,k)w(E){\|P\|}_{L^{\infty}(B)}, \quad P \in \pk.
		\end{equation}
		
	\end{lemma}
	
	\begin{proof}
		Suppose $P\in \pk$ with ${\|P\|}_{L^{\infty}(B)}=1$ but ${\|P\|}_{C^{k-1}(E)}\le cw(E)$. We will show that if $c=c(n,k)$ is chosen sufficiently small, the set $E$ lies entirely between two parallel hyperplanes at distance less than $\frac{1}{2}w(E)$, which is a contradiction. Take $x_0\in E$ and write the Taylor expansion of $P$ around this point then we have
		\[
		P(x)=\sum\limits_{|\alpha|\le k-1} \frac{\partial^{\alpha}P(x_0)}{\alpha!}(x-x_0)^\alpha + \sum\limits_{|\alpha|= k} \frac{\partial^{\alpha}P(x_0)}{\alpha!}(x-x_0)^\alpha,
		\]
		which implies the following inequality:
		\[
		{\|P\|}_{L^{\infty}(B)}\le c'(n,k)cw(E) +c'(n,k)\max\limits_{|\alpha|= k}|\partial^{\alpha}P(x_0)|.
		\]
		Noting that $w(E)\le2$, and for the moment choosing $c\le\frac{1}{4c'}$,  we get $|\partial^{\alpha}P(x_0)|\ge\frac{1}{2c'}$ for some $\alpha$ with $|\alpha|=k$. Now suppose $\alpha=(\alpha_1,...,\alpha_n)$ with $\alpha_i=\max\limits_{1\le j\le n }\alpha_j$ and let $\alpha'=(\alpha_1,..,\alpha_i-1..,\alpha_n)$. Then we note that in the derivative $\partial^{\alpha'}P$,  $x_i$ appears in one and only one term with the coefficient $\partial^{\alpha}P(x_0)$. So the polynomial $\partial^{\alpha'}P(x)$ has the form
		\[
		\partial^{\alpha'}P(x)=a\cdot x+b, \quad |a|\ge \frac{1}{2c'},
		\]
		and if we normalize it by setting $\nu=\frac{a}{|a|}$ and $b'=\frac{b}{|a|}$, for all $x\in E$ we get	
		\[
		|\nu \cdot x+b'|\le2c'cw(E),
		\]
		which means that $w(E,\nu)\le 4c'cw(E)$. So if we take $c(n,k)= \frac{1}{8c'}$ we get the desired contradiction and this finishes the proof.
	\end{proof}
	
	In the next two lemmas we show that how \eqref{hpc} can be relaxed so that we can replace $\|P\|_{C^{k-1}(E)}$ by a smaller quantity $\max\limits_{(x,\beta)\in F}|\partial^{\beta} P(x)|$, where $F$ is a set of pairs of points and indices with $\#F=\text{dim}(\pk)$. The first one is a general geometric result bounding the width of a compact set in $\rd$ with the width of its finite subsets with $d+1$ points.
	\begin{lemma}\label{rectanglelemma}
		Let $E$ be a bounded subset of $\rd$ and $\varepsilon>0$ such that for each $d$ points $x_1,x_2,...,x_d$ in $E$, there exists a unit vector $\nu$ satisfying $|\nu \cdot x_i|\le \varepsilon$ for $1\le i \le d$. Then there is a unit vector $\nu_0$ such that $|\nu_0 \cdot x|\le2d \varepsilon$ for every $x\in E$.
		
	\end{lemma}
	
	\begin{proof}
		First we note that the assumption on $E$ holds for $\bar{E}$ too, which allows us to assume $E$ is compact. This implies that there are $x_1,x_2,...,x_{d-1} \in E$ such that the $(d-1)$-dimensional area of the simplex $\Delta(0,x_1,..x_{d-1})$ is maximal. Let $A=|\Delta(0,x_1,..x_{d-1})|$ and note that if $A=0$, $E$ entirely lies in a $d-2$ dimensional subspace and $\nu_0$ can be taken as a unit vector perpendicular to that. Then suppose $A>0$ and let $\nu_0$ be the normal vector to $H=span(x_1,...,x_{d-1})$. We prove this vector has the desired property, which means that for any $x\in E$ we have to show that $|\nu_0 \cdot x|\le2d \varepsilon$ or equivalently $d(x,H)\le2d \varepsilon$.\\
		
		Now from the assumption there is a unit vector $\nu$ such that $|\nu \cdot x_i|\le \varepsilon$, $|\nu \cdot x|\le \varepsilon$, which implies that the $d$ dimensional simplex $\Delta(0,x,x_1,...,x_{d-1})$ lies between two parallel hyperplane of distance $2\varepsilon$ from each other. Consequently, the length of one of the altitudes of this simplex must be no more than $2\varepsilon$. Let us call the length of this altitude $h$ and the area of its corresponding face $A'$. There are two cases, either $0$ is one of the vertices of this face or not. In the former case $A'\le A$, by maximality of $A$, and in the latter, $A'\le dA$. The reason for this is the fact that the area of each face of a simplex is no more than the sum of areas of other faces, and all the other faces have $0$ as a vertex. Now that we know $A'\le dA$, we calculate the $d$ dimensional volume of the simplex $\Delta(0,x,x_1,...,x_{d-1})$ in two ways with different altitudes and faces, and we get $d(x,H)A=hA'\le2\varepsilon.dA$. Hence $d(x,H)\le2d\varepsilon$ and for each $x\in E$, $|\nu_0 \cdot x|\le 2d\varepsilon$.
	\end{proof}
	
	\begin{lemma}\label{findingF}
		Let $k$ be a positive integer, $B$ the unit ball of $\rn$ and $E\subset B$ satisfying
		\begin{equation}\label{polycond}
			{\|P\|}_{C^{k-1}(E)}> \delta {\|P\|}_{L^{\infty}(B)}, \quad P \in \mathbb{P}_k.
		\end{equation}
		Then there exists a set $F\subset\left\{(x,\beta)|x\in E, |\beta|\le k-1 \right\}$ with $\#F=\text{dim}(\pk)$ such that
		
		\begin{equation}\label{fcondition}
			\max\limits_{(x,\beta)\in F}|\partial^{\beta} P(x)|>c(n,k) \delta {\|P\|}_{L^{\infty}(B)}, \quad P \in \mathbb{P}_k.
		\end{equation}
		
	\end{lemma}
	\begin{proof}
		For $P \in \mathbb{P}_k$ with $P(y)=\sum\limits_{|\alpha|\le k}c_{\alpha}y^{\alpha}$ let $\hat{P}=(c_{\alpha})_{|\alpha|\le k}$. Here,  $\hat{P}$ is a $d$-dimensional vector with $d=\text{dim}\pk$, and we note that this defines an isomorphism between $\pk$ and $\rd$. Then from the finiteness of the dimension we have $\|\hat{P}\|\approx {\|P\|}_{L^{\infty}(B)}$, where $\|\hat{P}\|$ is the Euclidean norm in $\rd$. So in \eqref{polycond} we may replace ${\|P\|}_{L^{\infty}(B)}$ by $\|\hat{P}\|$, and $\delta$ by $c'(n,k) \delta$. Then for each $x\in E$ and $|\beta|\le k-1$ we set $\hat{(x,\beta)}=(b_{\alpha})_{|\alpha|\le k}$, where this vector is defined by
		
		\[
		b_\alpha =\begin{cases}
			\frac{\alpha!}{(\alpha-\beta)!}x^{\alpha-\beta}& \beta \le \alpha\\
			0& \text{otherwise}
		\end{cases},
		\]
		and it satisfies
		
		\[
		\partial^{\beta}P(x)=\left\langle\hat{P},\hat{(x,\beta)} \right\rangle, \quad P \in \mathbb{P}_k.
		\]
		Now consider the set 
		$$\hat{E}=\left\{\hat{(x,\beta)}|x\in E, |\beta|\le k-1 \right\},$$
		which is bounded in $\rd$, and note that if for every $d$ points like $Q_1,..,Q_d$ in $\hat{E}$ there exists a unit vector $\hat{P}$ with $|\left\langle \hat{P}, Q_i \right\rangle|\le \frac{c'(k,n)}{2d}\delta$, then according to Lemma \ref{rectanglelemma} there is a unit vector $\hat{P}$ such that $|\left\langle \hat{P}, Q \right\rangle|\le c'(k,n)\delta$ for each $Q\in \hat{E}$. This means that $\|P\|_{C^{k-1}(E)}\le c'(k,n)\delta$, which is a contradiction to \eqref{polycond}. So there exists $d$ points $Q_1,..,Q_d$ in $\hat{E}$ such that
		
		\[
		\inf_{\|\hat{P}\|=1}\max\limits_{1\le i \le d} |\left\langle \hat{P}, Q_i \right\rangle|> \frac{c'(k,n)}{2d}\delta,
		\]
		and this means that there is a set $F\subset\left\{(x,\beta)|x\in E, |\beta|\le k-1 \right\}$ with $\#F=\text{dim}(\pk)$ such that
		
		\[\max\limits_{(x,\beta)\in F}|\partial^{\beta} P(x)|>\frac{c'(k,n)}{2d}\delta {\|\hat{P}\|}, \quad P \in \mathbb{P}_k.
		\]
		Now it's enough to replace $\|\hat{P}\|$ by $\|P\|_{L^{\infty}(B)}$ and get \eqref{fcondition}.
	\end{proof}
	
	\begin{lemma}\label{interpolationlemma}
		Let $E$ be a subset of $B$, the unit ball of $\rn$, with $w(E)>0$. Then there exists a set $F\subset\left\{(x,\beta)|x\in E, |\beta|\le k-1 \right\}$ with $\#F=\text{dim}(\pk)$, and a basis of polynomials $\left\{P_{(x,\beta)}\right\}$ for $\pk$ such that for all $(x',\beta')\in F$ we have
		\begin{equation}\label{interpolationeq}
			\partial^{\beta'}(P_{(x,\beta)})(x')=\begin{cases}
				1 & (x',\beta')=(x,\beta)\\
				0 & (x',\beta')\ne(x,\beta)
			\end{cases} 
		\end{equation}
		Moreover, 
		\begin{equation}\label{estimate for p}
			\|P_{(x,\beta)}\|_{L^{\infty}(B)}\lesssim w(E)^{-1} \quad (x,\beta)\in F.
		\end{equation}
	\end{lemma}

	\begin{proof}
		From Lemmas \ref{ckhyperlemma} and \ref{findingF} it follows that there exists a set $F$ with $\#F=\text{dim}(\pk)$ such that for every $P \in \pk$ we have
		\[
		{\|P\|}_{L^{\infty}(B)} \lesssim w(E)^{-1}\max\limits_{(x,\beta)\in F}|\partial^{\beta} P(x)|.
		\]
		This implies that the mapping $P\longrightarrow (\partial^\beta P(x))_{(x,\beta)\in F}$ is an isomorphism from $\pk$ to $\mathbb{R}^{\text{dim}(\pk)}$, therefore for each $(x,\beta)\in F$, \eqref{interpolationeq} has a unique solution $P_{x,\beta}$. These polynomials satisfy \eqref{estimate for p}, and since $\#F=\text{dim}(\pk)$, they form a basis for $\pk$, and this finishes the proof.	 
	\end{proof}
	
	Our last lemma in this section is a simple generalization of boundedness of norms of atoms in $\hprn$ (see \cite{MR1232192} Ch.3 Further results 5.5).
	\begin{lemma}\label{atomlemma}
		Suppose $B$ is the unit ball of $\rn$, $f$ a distribution supported in $B$ and of the form $f=g+\sum\limits_{(x,\beta)\in F}c_{x,\beta}\partial^{\beta}\delta_x$ with $\infn{g}\le1$, $|c_{x,\beta}|\le 1$, and $F$ a nonempty finite set of pairs $(x,\beta)$ such that $x\in B$ and $|\beta|<n(1/p -1)$. Moreover, suppose all moments of $f$ vanishes up to order $N_p=[n(1/p -1)]$. Then $f\in \hprn$ and $\hprnn{f}\lesssim \#F$.
	\end{lemma}
	
	\begin{proof}
		
		Let $\varphi \in \cinftyc{B}$, $0\le \varphi \le 1$ with $\int\varphi=1$. We estimate $M_\varphi(f)$ by following the usual local-nonlocal argument. For $x_0 \in 3B$ we have
		
		\[
		\varphi_t*f(x_0)=\varphi_t*g(x_0)+ \sum\limits_{(x,\beta)\in F}(-1)^\beta c_{x,\beta}t^{-n-|\beta|}(\partial^{\beta}\varphi)(\frac{x_0-x}{t})	= \textrm{I}+\textrm{II}.
		\]  
		Now for $t>0$ , $|\textrm{I}|\le 1$, and we note that in the second term $(\partial^{\beta}\varphi)(\frac{x_0-x}{t})\neq 0$ only if $t>|x-x_0|$, so this term is bounded by
		
		\[
		|\textrm{II}|\lesssim \sum\limits_{(x,\beta)\in F} |x-x_0|^{-n-|\beta|}.
		\]
		This gives us an estimate for the local part and it remains to estimate $M_\varphi(f)$ on $(3B)^c$, where we have to use cancellation properties of $f$. So this time assume $x_0\notin 3B$ and note that since $\varphi_t(x_0-\cdot)$ is supported in $B_t(x_0)$, $f*\varphi_t(x_0)\neq0$ only if $t>|x_0|/2$. To use the cancellation properties of $f$, let $P$ be the Taylor polynomial of $\varphi_t(x_0-\cdot)$ around $0$ and of order $N_p$. Then we have
		\begin{align*}
			&f*\varphi_t(x_0)=\left\langle f, \varphi_t(x_0-\cdot)\right\rangle=\left\langle f, \varphi_t(x_0-\cdot)-P \right\rangle\\
			&=\int g(y)(\varphi_t(x_0-y)-P(y))dy+\sum\limits_{(x,\beta)\in F}(-1)^\beta c_{x,\beta} \partial^\beta(\varphi_t(x_0-y)-P(y))|_{y=x}\\
			&=\textrm{I}+\textrm{II}.
		\end{align*}
		Now from the Taylor remainder theorem it follows that for $y\in B$ we have $$|\varphi_t(x_0-y)-P(y)|\le \infn{\nabla^{N_p+1}(\varphi_t)}\lesssim t^{-(N_p+1+n)},$$ 
		therefore $|\textrm{I}|\lesssim |x_0|^{-(N_p+1+n)}$. To estimate the second term we note that $\partial^{\beta}P$ is the $(N_p-|\beta|)$-th order Taylor polynomial of $\partial^\beta(\varphi_t(x_0-\cdot))$ around $0$, and this gives us	
		\[
		|\partial^\beta(\varphi_t(x_0-y)-P(y))|\lesssim \infn{\nabla^{N_p-|\beta|+1}(\partial^\beta(\varphi_t))}\lesssim t^{-(N_p+1+n)}\lesssim |x_0|^{-(N_p+1+n)},
		\]
		so we have $|\textrm{II}|\lesssim |x_0|^{-(N_p+1+n)}$. Next, by bringing the estimates for the local and nonlocal parts together we obtain
		
		\[
		M_{\varphi}(f)(x_0)\lesssim \left(1+ \sum\limits_{(x,\beta)\in F} |x-x_0|^{-n-|\beta|}\right)\mathbbm{1}_{3B}(x_0)+|x_0|^{-(N_p+1+n)}\mathbbm{1}_{(3B)^c}(x_0).
		\] 
		Now using the above estimate and our assumptions on $\beta$ and $N_p$ gives us
		\[
		\hprnn{f}=\lprnn{M_\varphi(f)}\lesssim \#F,
		\]
		which completes the proof.
	\end{proof}
	
	With the above lemmas in hand we are ready to prove the sufficiency part of the Theorem \ref{theorem1}.
	
	\begin{proof}[Proof of sufficiency] According to Lemma \ref{lemma0} it is enough to extend each $(p,\Omega)$-atom to an element of $\hprn$ with a uniform bound on its norm.\\
		
		Case $p=1$. Take a $(1,\Omega)-$ atom $g$ supported on a ball $B_r(x_0)$, with $r\approx d(x_0)$. Now, from the assumption there are positive constants $a$ and $\delta$ such that \eqref{h1cc} holds. Then we extend $g$ by putting a constant on $\Omega^c\cap B_{ar}(x_0)$, to create the required cancellation. So let
		\[
		\tilde{g}=g-c.\mathbbm{1}_{\Omega^c \cap B_{ar}(x_0)}, \quad c= \frac{\int g}{|\Omega^c \cap B_{ar}(x_0)|}.
		\]
		Here, we note that $\text{supp}(\tilde{g})\subset B_{ar}(x_0)$, $\int\tilde{g}=0$ and this function is bounded by
		
		\[
		\infn{\tilde{g}}\le \infn{g}\left(1+a^{-n}\frac{|B_{ar}(x_0)|}{|\Omega^c \cap B_{ar}(x_0)|} \right)\lesssim\infn{g}\left(1+a^{-n}\delta^{-1}\right)\lesssim_{a,\delta}|B_{ar}(x_0)|^{-1}.
		\]
		Hence $\tilde{g}$ is a constant multiple of a 1-atom so $\hrnn{\tilde{g}}\le c(a,\delta,n)$, and this finishes the proof of this case.\\
		
		Case $p=\frac{n}{k+n}$, $k=1,2,...$. Let $g$ be a $(p,\Omega)$-atom similar to the previous case, and let $a>1$, $\delta>0$ be the constants that satisfy \eqref{key}. By a translation and dilation, we bring everything to the unit ball and set	
		\[
		E=\frac{\left(\Omega^c\cap \bar{B}_{ar}\right)-x_0}{ar}, \quad g'=g(x_0+arx).
		\]
		Now	$E$ is a subset of the unit ball $B$ with $w(E)>a^{-1}\delta$, so let $F$ be a set provided by Lemma \ref{interpolationlemma}. Then we extend $g'$ by
		\[
		\widetilde{g'}=g'+\sum\limits_{(x,\beta)\in F}c_{x,\beta}\partial^{\beta}\delta_x,
		\]
		where the coefficient $c_{x,\beta}$ are chosen in way that $\widetilde{g'}$ has vanishing moments up to order $k$. To see this can be done, recall that from Lemma \ref{interpolationlemma} we may choose a basis of polynomials for $\pk$  $\left\{P_{x,\beta}\right\}$, such that $\partial^\beta (P_{x,\beta})(x)=1$ and $\partial^{\beta'} (P_{x,\beta})(x')=0$ when $(x,\beta)\neq(x', \beta')$. Then the cancellation condition for $\widetilde{g'}$ is equivalent to
		\[
		\left\langle\widetilde{g'},P_{x,\beta}\right\rangle=0 \quad (x,\beta)\in F,
		\]
		which gives us the following equation for the coefficients $c_{x,\beta}$:
		\begin{equation}\label{solvingc}
			\int g'(y)P_{x,\beta}(y)dy+\sum\limits_{(x',\beta')\in F}c_{x',\beta'}\left\langle\partial^{\beta'}\delta_{x'},P_{x,\beta}\right\rangle=0 \quad (x,\beta)\in F.
		\end{equation}
		Now we note that 
		\begin{equation}\label{cxbeta}
			\left\langle\partial^{\beta'}\delta_{x'},P_{x,\beta}\right\rangle=(-1)^{|\beta'|}\partial^{\beta'}(P_{x,\beta})(x'),
		\end{equation}
		so from \eqref{solvingc}, \eqref{cxbeta} and Lemma \ref{interpolationlemma} we have
		\begin{equation}\label{cxbetaequation}
			c_{x,\beta}=(-1)^{|\beta|+1}\int g'(y)P_{x,\beta}(y)dy.
		\end{equation}
		This shows that $\widetilde{g'}$ depends linearly on $g$, $\left\langle \tilde{g'},P	
		\right\rangle =0$ for every $P \in \pk$, and moreover 
		
		\begin{align*}
			&|c_{x,\beta}|\le |B|\linfno{g'}{B}\linfno{P_{x,\beta}}{B}\lesssim
			\linfno{g'}{B} w(E)^{-1}\lesssim a\delta^{-1} \linfno{g'}{B}.
		\end{align*}		
		Now $\widetilde{g'}/\linfno{g'}{B}$ satisfies all the required conditions of Lemma \ref{atomlemma} and by applying that we get
		
		\[
		\hprnn{\tilde{g'}}\lesssim \linfno{g'}{B} \lesssim_{a,\delta} |B_{ar}(x_0)|^{\frac{-1}{p}}.
		\]		
		Finally, let $\tilde{g}=\widetilde{g'}(\frac{x-x_0}{ar})$ which is an extension of $g$	with $\hprnn{\tilde{g}}\le c(n,k,a,\delta)$, and this completes the proof of the second case.\\
		
		Case $p\neq\frac{n}{n+k}$, $k=0,1,2,...$. This case is similar the previous one and even simpler. This time we can take $a$ to be any number such that $E={\Omega}^c\cap B_{ar}\ne\emptyset$ and after a dilation and translation, we pick a point $x\in E$ and extend $g'$ by 
		\[
		\widetilde{g'}=g'+\sum\limits_{|\beta|\le [n(1/p-1)]}c_{\beta}\partial^\beta \delta_x, \quad c_\beta=(-1)^{|\beta|+1}\int g'(y)\frac{(y-x)^\beta}{\beta!} dy.
		\]
		The rest is just like the previous case.\\
		
		In all the above cases, the method of extension is linear in terms of $g$ so from Lemma \ref{lemma0}, there is a linear extension operator if the conditions of Theorem \ref{theorem1} hold.
	\end{proof}
	
	\section{Proof of necessity}
	Now that we've proved the conditions of Theorem \ref{theorem1} are sufficient, we have to show they are necessary as well. To do this, we need to construct some functions in the dual of $\hprn$ to prevent any bounded extension when those conditions are not satisfied. We need some lemmas and our first one is a simple consequence of a theorem proved by Uchiyama. We give it below and its proof is found in \cite{MR0651495,MR0658065}.
	
	\begin{theorem*}[Uchiyama]
		Let $\lambda>0$ and $E$ , $B$ be two subsets of $\rn$ such that for any cube $Q$ we have
		
		\begin{equation}\label{uchi}
			\min\left\{\frac{|Q\cap E|}{|Q|},\frac{|Q\cap B|}{|Q|}\right\} \le 2^{-2n\lambda}.
		\end{equation}
		Then there exists a function $f\in \bmorn$ such that $f=1$ on $B$ and $f=0$ on $E$ with $\bmornn{f}\lesssim\frac{1}{\lambda}$.
	\end{theorem*}
	
	\begin{lemma}\label{bmolemma}
		Let $0<\varepsilon\le \frac{1}{2}$, $a>3$ and $B$ the unit ball of $\rn$. Then for every set $E\subset (2B)^c $ such that $|E\cap aB|<\varepsilon$, there exists a function $f\in \bmorn$ with the following properties:
		\begin{itemize}
			\item $f=0$  on $E$
			\item $f=1 $ on $B$
			\item $\bmornn{f}\lesssim \frac{1}{\min\left\{\log a, \log\varepsilon^{-1}\right\}}$
		\end{itemize}
	\end{lemma}
	
	\begin{proof}
		Let us check the condition of the above theorem for $B$ and $E$. If a cube $Q$ doesn't intersect one of the sets \eqref{uchi} holds for any $\lambda>0$, so assume it intersects both of them. Then there are two possibilities either $Q\cap(aB)^c\neq\emptyset$ or $Q\subset aB$. In the first case $l(Q)\ge\frac{(a-2)}{\sqrt{n}}$ which implies $\frac{|Q\cap B|}{|Q|}\lesssim a^{-n}$, and in the second one we have $l(Q)\ge 1$, which gives us $\frac{|Q\cap A|}{|Q|}= \frac{|Q\cap E\cap aB|}{|Q|} \le \varepsilon$. So for $\lambda \approx \min\left\{\log a,\log\varepsilon^{-1}\right\}$ \eqref{uchi} holds and the claim follows from the theorem.
	\end{proof}
	
	To construct similar functions of the above lemma in $\hlip{k}{n}$, we need a couple of facts which are well known and we bring here as lemmas. Here it's convenient to write $\text{BMO}(\mathbb{R})=\dot{\Lambda}^{0}(\mathbb{R})$.
	
	\begin{lemma}\label{intlemma}
		Let $f\in\dot{\Lambda}^{k}(\mathbb{R})$ and  $F(x)=\int_{0}^{x}f(t)dt$,  Then $F \in\dot{\Lambda}^{k+1}(\mathbb{R}) $ with $\|F\|_{\dot{\Lambda}^{k+1}(\mathbb{R})} \lesssim \|f\|_{\dot{\Lambda}^{k}(\mathbb{R})}$ for $k=0,1,2,...$.
	\end{lemma}
	
	\begin{proof}
		For $k=0$ we have
		
		\[
		|\Delta^2_hF(x)/h|=\left|1/h\int_{x+h}^{x+2h}f-1/h\int_{x}^{x+h}f\right|\lesssim\|f\|_{\text{BMO}(\mathbb{R})},
		\]
		and for $k\ge1$ we have
		
		\[
		\Delta^{k+2}_hF(x)=\Delta^{k+1}_h(\Delta^1_hF)(x)=\Delta^{k+1}_h\int_{0}^{h}f(x+t)dt=\int_{0}^{h}(\Delta^{k+1}_hf)(x+t)dt,
		\]
		which implies 
		\[
		|\frac{\Delta^{k+2}_hF(x)}{h^{k+1}}|\le \frac{1}{|h|} \int_{0}^{h} |\frac{\Delta^{k+1}_h}{h^k}|(x+t)dt\le\|f\|_{\dot{\Lambda}^{k}(\mathbb{R})}.
		\]
		So in any case we get $\|F\|_{\dot{\Lambda}^{k+1}(\mathbb{R})} \lesssim \|f\|_{\dot{\Lambda}^{k}(\mathbb{R})}$.
	\end{proof}
	
	\begin{lemma}\label{cuttinglemma}
		Let $k$ be a positive integer, $f\in \hlip{k}{n}$ and $g\in \ckrn$. Then $fg\in \hlip{k}{n}$ and $\hlipn{fg}{k}{n} \lesssim \sum_{l=1}^{k}\hlipn{f}{l}{n}\infn{\nabla^{k-l}(g)}+\infn{f}\infn{\nabla^{k}(g)}$.
	\end{lemma}
	
	\begin{proof}
		Let $\tau^{h}g(x)=g(x+h)$, so for $\Delta^1_h(fg)$ we have
		\begin{equation}\label{delta1}
			\Delta^1_h(fg)=\Delta^1_h(f)\tau^{h}g+ f \Delta^1_h(g).
		\end{equation}
		By taking the difference operator $k+1$ times and applying \eqref{delta1} each time we can see that $\Delta^{k+1}_h(fg)$ has the form
		\begin{equation}\label{diffeerncekfg}
			\Delta^{k+1}_h(fg)= \sum\limits_{-1\le l \le k}\sum\limits_{0\le m,s\le k+1} c_{l,k,m,s}\tau^{mh} \Delta^{l+1}_h(f)\tau^{sh}\Delta^{k-l}_h(g),
		\end{equation}
		where in this formula $\Delta^0_h(f)=f$ and $c_{l,k,m,s}$ are some integers. Now for $1\le l\le k$ we have 
		\[
		|\Delta^{l+1}_h(f)|\le\|f\|_{\dot{\Lambda}^{l}(\mathbb{R})}|h|^l, \quad |\Delta^{k-l}_h(g)|\lesssim \infn{\nabla^{k-l}(g)}|h|^{k-l},
		\]
		where $\infn{\nabla^0(g)}=\infn{g}$,
		and for $l=-1,0$ we have
		\[
		|\Delta^{l+1}_h(f)|\lesssim \infn{f}, \quad |\Delta^{k-l}_h(g)|\lesssim \infn{\nabla^{k}(g)}|h|^{k}.
		\]
		Putting these estimates in \eqref{diffeerncekfg} completes the proof.
		
	\end{proof}
	One more property of $\hlip{k}{n}$ that is needed here is the inequality
	\begin{equation}\label{interpolationinequality}
		\hlipn{f}{l}{n} \lesssim \infn{f}^{1-\frac{l}{k}}\hlipn{f}{k}{n}^{\frac{l}{k}} \quad 1\le l \le k,
	\end{equation} 
	which is a consequence of the fact that $\hlip{l}{n}$ is the interpolation space between $L^{\infty}(\rn)$ and $\hlip{k}{n}$ (see \cite{MR4567945} Theorem 17.30).
	
	\begin{lemma}\label{lip1lemma}
		Suppose $0<\varepsilon<1/4$, $a>4$ and let $I\subset [-a,a]$ be an interval with $|I|<2\varepsilon$. Then for each $k=1,2,3,...$ there exists a function $f_k \in \hlip{k}{1}$ with the following properties:
		
		\begin{itemize}
			\item $f_k=0$ on $I\cup[-a,a]^c$
			\item $|f_k|\gtrsim \min\{\log a, \log \varepsilon^{-1}\}$ on  $J=[x_0-1/2,x_0+1/2]^c\cap[-1,1]$
			\item $\infn{f_k}\lesssim a^k$ 
			\item $\hlipn{f_k}{k}{1}\lesssim 1$
		\end{itemize}
	\end{lemma}
	
	\begin{proof}
		Suppose $x_0$ is the mid point of $I$ and consider the function $$g(t)=\min\left\{\log^+|\frac{t-x_0}{\varepsilon}|,\log^+|\frac{a}{t}|\right\},$$
		which is nonnegative, vanishes on $I\cup [-a,a]^c$, $g\ge \min \{\log a, \log\frac{\varepsilon^{-1}}{4}\}$ on the interval $[x_0-1/4,x_0+1/4]^c\cap[-1,1]$ and also $\|g\|_{\text{BMO}}\lesssim 1$. Now let $F_1(x)=\int_{x_0}^{x}g(t)dt$ which vanishes on $I$ and  $|F_1|\gtrsim\min \{\log a, \log\varepsilon^{-1}\} $ on $J=[x_0-1/2,x_0+1/2]^c\cap[-1,1]$. Also, according to Lemma \ref{intlemma} we have $\hlipn{F_1}{1}{1}\lesssim 1$ and we note that
		\[
		|F_1(x)|\le \int_{-a}^{a} g(t)dt\le\int_{-a}^{a} \log^+|\frac{a}{t}|dt\lesssim a.
		\]
		Now let $\varphi\in \cinftyc{[-1,1]}$ such that $0\le \varphi \le 1$, $\varphi=1$ on $[-1/2,1/2]$ and set $\varphi_a(x)=\varphi(x/a)$. Then the function $f_1(x)=F_1(x)\varphi_a(x)$ has the following properties:
		\begin{itemize}
			\item $f_1=0$ on $I\cup[-a,a]^c$
			\item $f_1$ is positive on the right hand side of $x_0$ and negative on the other side  with $|f_1|\gtrsim\min \{\log a, \log\varepsilon^{-1}\}$ on $J$
			\item $\infn{f_1}\lesssim a$
			\item $\hlipn{f_1}{1}{1}\lesssim 1$
		\end{itemize}
		For the last property it's enough to note that
		$$\hlipn{F_1}{1}{1}\lesssim 1, \quad \infn{F_1}\lesssim a, \quad \infn{\varphi_a}\le1, \quad \infn{\varphi_a'}\lesssim a^{-1},$$
		and therefore Lemma \ref{cuttinglemma} implies $\hlipn{f_1}{1}{1}\lesssim 1$. We continue this process, meaning each time we integrate the function of the previous step from $x_0$ and with $\varphi_a$ cut it off the interval $[-a,a]$ then at k-th step we have  $F_k(x)=\int_{x_0}^{x}f_{k-1}(t)dt$ and $f_k=F_k\varphi_a$.\\
		
		The function $F_k$ vanishes on $I$, is positive on the right hand side of $x_0$ and negative or positive (depending on $k$) on the other side. Also, it follows from Lemma \ref{intlemma} that
		\begin{equation}\label{Fklipnorm}
			\hlipn{F_k}{k}{1}\lesssim 1,
		\end{equation}
		and moreover
		\begin{equation}\label{fksupnorm}
			\infn{F_k}\le \int_{-a}^{a}|f_{k-1}|\lesssim a^k.
		\end{equation}
		Now from (\ref{interpolationinequality}-\ref{fksupnorm}) we get $\hlipn{F_k}{l}{1}\lesssim a^{k-l}$ for $1\le l\le k$. We also have the bound $\infn{\varphi_a^{(k-l)}}\lesssim a^{l-k}$ for $0\le l\le k$. Now these facts combined with Lemma \ref{cuttinglemma} shows that $\hlipn{f_k}{k}{1}\lesssim 1$.
		Finally we observe that at the previous step, $f_{k-1}$ changes sign only at $x_0$, which implies that by integrating it from $x_0$ no cancellation happens and then $|f_k|\gtrsim\min \{\log a, \log\varepsilon^{-1}\}$ on $J$. The proof is now complete.
	\end{proof}
	
	The following is an analog of Lemma \ref{bmolemma} in the $\hlip{k}{n}$ setting.
	
	\begin{lemma}
		Let $0<\varepsilon< 1/4$, $a>4$ and $B$ the unit ball of $\rn$. Then for every set $E\subset\rn$ with $w(E\cap aB) <\varepsilon$, there exists a function $f\in \hlip{k}{n}$ such that
		\begin{itemize}\label{lipslemmaa}
			\item $f=0$  on $E$
			\item $\int_{B}|f|\approx 1 $
			\item $\hlipn{f}{k}{n}\lesssim \frac{1}{\min\left\{\log a, \log\varepsilon^{-1}\right\}}$
			
		\end{itemize}
	\end{lemma}
	\begin{proof}
		Let $\nu$ be a direction for which $w(E\cap aB,\nu) \le\varepsilon$ and choose $c$ such that $|\nu \cdot x+c|\le \varepsilon$ for $x\in E\cap aB $. Also let $\varphi\in\cinftyc{B}$ with $0\le \varphi \le 1$ and $\varphi=1$ on $\frac{1}{2}B$. Now take the function $f_k\in \hlip{k}{1}$ constructed in the previous lemma with $I=\left[-\varepsilon,\varepsilon\right]$ and set 
		\[
		f(x)=\frac{1}{\min\left\{\log a, \log\varepsilon^{-1}\right\}}f_k(\nu \cdot x+c)\varphi(x/a).
		\]
		This function has the desired properties. In fact the first two properties of $f$ follows directly from properties of $f_k$, and for the third one we note that $\hlipn{f_k(\nu \cdot +c)}{k}{n}=\hlipn{f_k}{k}{1}\lesssim 1$, so an application of Lemma \ref{cuttinglemma} gives us $\hlipn{f}{k}{n}\lesssim \frac{1}{\min\left\{\log a, \log\varepsilon^{-1}\right\}}$.
	\end{proof}
	
	Now that we have constructed our functions in the dual of Hardy spaces we can prove the necessity part of Theorem \ref{theorem1}.
	
	\begin{proof}[Proof of necessity] We have two cases:\\
		
		Case $p=1$. Suppose $\Omega$ doesn't satisfy \eqref{h1cc} then for $\delta_j=2^{-j}$ and $a_j=j$ there's a point $x_j$ such that \eqref{h1cc} doesn't hold. Let us denote $B_j=B(x_j,\frac{1}{2}d(x_j))$ and $r_j=\frac{1}{2}d(x_j)$.\\
		
		Now suppose $\Omega$ still is an extension domain and take a sequence of $(1,\Omega)$-atoms $g_j$ with supporting ball $B_j$ and denote their extensions by $\tilde{g_j}$. Then we must have $\hrnn{\tilde{g_j}}\lesssim 1$. Now by a translation and dilation we move every thing to the unit ball and set 
		\[
		E_j=\frac{\Omega^c -x_j}{r_j}, \quad g'_j(x)=r_j^{n}g_j(r_jx+x_j),\quad \tilde{g_j}'(x)=r_j^{n}\tilde{g_j}(r_jx+x_j).
		\]
		Here we note that $E_j$ satisfies the condition in Lemma \ref{bmolemma} with $\varepsilon=\varepsilon_j\approx\delta_ja_j^{n}$ and $a=a_j$. So there exists a function $f_j\in \bmorn$ with properties of Lemma \ref{bmolemma}. Now from duality and the fact that $\hrnn{\tilde{g_j}'}\lesssim 1$ we get
		
		\[
		| \left\langle \tilde{g_j}',f_j	
		\right\rangle |\lesssim\hrnn{\tilde{g_j}'}\|f_j\|_{\bmorn} \lesssim \frac{1}{\min\left\{\log a_j, \log\varepsilon_j^{-1}\right\}}
		\]
		and on the other hand since $f_j=0$ on $E_j$ we have
		\begin{equation}\label{fj=0}
			\left\langle \tilde{g_j}',f_j	
			\right\rangle=\int_{B_1}g_j'f_j,
		\end{equation}
		which implies the following contradiction if we choose $g_j=\text{sgn}(f_j).|B_j|^{-1}$
		
		\[
		1\approx\int_{B_1}|f_j|\approx |\left\langle \tilde{g_j}',f_j	
		\right\rangle| \lesssim \frac{1}{\min\left\{\log a_j, \log\varepsilon_j^{-1}\right\}} \longrightarrow 0.
		\]
		The proof of this case is now complete.\\
		
		Case $p\ne 1$. The above argument holds word by word if instead of Lemma \ref{bmolemma} we use Lemma \ref{lipslemmaa} and note that the functions $f_j$ vanishes in a neighborhood of $E_j$ so \eqref{fj=0} holds.\\
		This finishes the proof of necessity.
	\end{proof}
	\section{Applications}
	As an application of Theorem \ref{theorem1} we are able to characterize the dual of $H^1(\Omega)$  as a complemented subspace of $\bmorn$ (see \cite{MR1060724} the final remark). Let us recall that the dual space of $H^1(\Omega)$ which is denoted by $BMO(\Omega)$ in \cite{MR1060724}, consists of those functions $f$ such that
	\[
	\|f\|_{BMO(\Omega)}= \sup_{2Q\subset \Omega} \dashint_Q|f-\dashint_Qf| + \sup_{Q\in \mathcal{W}(\Omega)}\dashint_Q|f| < \infty,
	\]
	where $\mathcal{W}(\Omega)$ is the set of all Whitney cubes of $\Omega$. By extending $f\in BMO(\Omega)$ to be zero outside of $\Omega$, we get $\tilde{f}\in\bmorn$ with $\bmornn{\tilde{f}}\lesssim \|f\|_{BMO(\Omega)}$. The converse of this inequality is also true when $\Omega$ is an $H^1$-extension domain. Indeed by using duality we have
	\[
	\|f\|_{BMO(\Omega)}\approx\sup_{\|g\|_{H^1(\Omega)}=1} \left\langle f,g \right\rangle \lesssim_{a,\delta} \sup_{\|\tilde{g}\|_{H^1(\rn)}=1} \left\langle \tilde{f},\tilde{g} \right\rangle \approx \bmornn{\tilde{f}}.
	\]
	The above is true because every $g\in H^1(\Omega)$ has a bounded extension $\tilde{g} \in H^1(\rn)$. So we have the following:
	\begin{corollary}
		When $\Omega$ is an extension domain for $H^1$, the dual of $H^1(\Omega)$ is isomorphic to the closed subspace of $\bmorn$, consisting of all functions with support on $\Omega$.
	\end{corollary}
	Next, let $E:H^1(\Omega) \longrightarrow H^1(\rn)$ be the extension operator,
	then by duality the adjoint operator $E':\bmorn \longrightarrow BMO(\Omega)$ given by
	\[
	\left\langle E(f),g \right\rangle = \left\langle f,E'(g) \right\rangle, 
	\]
	is bounded. Now if $g\in \bmorn$ with support on $\Omega$ and $f$ is any $(1,\Omega)$-atom supported on a Whitney cube $Q$, we have
	\[
	\int fg=\int E(f)g = \int fE'(g).
	\]
	So on each Whitney cube $Q$, $E'(g)=g$, and therefore $E'(g)=g$ on the whole $\Omega$. This allows us to conclude the following:
	
	\begin{corollary}
		Let $\Omega$ be an open subset of $\rn$ satisfying \eqref{h1cc}. Then the subspace of all functions in $\bmorn$, vanishing outside of $\Omega$ is a complemented subspace of $\bmorn$. Moreover, the operator $E'$ defined above is a bounded projection onto this subspace.
	\end{corollary}

	\section*{Acknowledgements}
	I'm grateful to Professor G. Dafni for suggesting the problem and valuable discussions and comments. Also, I'm is indebted to the referees for their valuable comments on the manuscript, especially for bringing my attention to the global Markov condition.
	
	\normalsize
	
	\bibliographystyle{plain}
\bibliography{mybib}

\end{document}